\documentclass[11pt]{article}
\usepackage{graphicx} 
\usepackage{a4wide}

\usepackage{amsmath}

\usepackage{color}
\usepackage{amsmath,amssymb,enumerate}
\usepackage{pstricks}

\usepackage{epsfig}

\usepackage{amsmath}   
\usepackage{amssymb} 

\usepackage{hyperref}
\usepackage{dsfont}
\usepackage[all]{xy}

\usepackage{pgfplots}
\pgfplotsset{compat=1.15}
\usepackage{mathrsfs}
\usetikzlibrary{arrows}
\usetikzlibrary{patterns}
\usepackage{xcolor}
\usepackage{tikz}
\usetikzlibrary{calc,trees,positioning,fit,shapes,calc}

\usepackage{graphics}

\usepackage{pst-3d,float}

\newtheorem{theorem}{Theorem}[section]

\newtheorem{corollary}[theorem]{Corollary}
\newtheorem{definition}[theorem]{Definition}
\newtheorem{lemma}[theorem]{Lemma}
\newtheorem{proposition}[theorem]{Proposition}
\newtheorem{remark}[theorem]{Remark}
\newenvironment{proof}{\medskip\noindent{\bf Proof.}\;}{\null\hfill $\Box$\par\medskip }

\newcommand{\supp}{{\rm supp}}
\newcommand{\diam}{{\rm diam}}
\newcommand{\C}{\ensuremath{\mathcal C}}
\newcommand{\I}{\ensuremath{\mathcal I}}
\newcommand{\K}{\ensuremath{\mathcal K}}
\newcommand{\Ke}{\ensuremath{\mathcal K}_{\varepsilon}}
\newcommand{\Ken}{\ensuremath{\mathcal K}_{\varepsilon_n}}
\newcommand{\Zen}{\ensuremath{\mathcal Z}_{\varepsilon_n}}

\newcommand{\eps}{\ensuremath{\varepsilon}}
\newcommand{\epsn}{\ensuremath{\varepsilon_n}}

\title{Constructing self-similar subsets within the fractal support of Lacunary Wavelet Series for their multifractal analysis}

\author{C\'eline Esser$^a$  and B\'eatrice Vedel$^b$\\\\
$^a$ Universit\'e de Li\`ege,  All\'ee de la D\'ecouverte 12, B-4000
Li\`ege, Belgium \\
$^b$Universit\'e Bretagne Sud, CNRS UMR 6205, LMBA, F-56000 Vannes, France}         

\date{ }

\def \Z {\mathbb{Z}}
\def \N {\mathbb{N}}
\def \R {\mathbb{R}}
\def \C {\mathcal{C}}

\begin{document}

\maketitle
\begin{abstract}
Given a fractal $\I$ whose Hausdorff dimension matches with the upper-box dimension, we propose a new method  which consists in selecting inside  $\I$  some subsets (called quasi-Cantor sets) of almost same dimension and with controled properties of self-similarties at prescribed scales. It allows us to estimate below the Hausdorff dimension $\I$ intersected to limsup sets of contracted balls selected according a Bernoulli law,  in contexts where classical Mass Transference Principles cannot be applied. We apply this result to the computation of the increasing multifractal spectrum of lacunary wavelet series supported on $\I$.
\end{abstract}



\section{Introduction and statement of the main results : alternative}

Multifractal analysis has become a pivotal tool for analyzing irregular signals. This field has been extensively studied across a wide range of functions. It provides a useful framework for examining the structures of sets and functions, giving interesting perspectives on many mathematical areas like stochastic analysis, number theory, ergodic theory and functional analysis for example. One of the key challenges in multifractal analysis is determining the Hausdorff dimension of subsets of $\mathbb{R}$ (or $\mathbb{R}^d$), known as iso-H\"older sets. These sets correspond to the level sets of the local regularity exponents of functions. Often, these subsets are described as limsup sets of balls, a connection that naturally arises from the characterization of regularity based on wavelet coefficients, which involves a liminf condition (see Theorem \ref{thm:waveletcharact} and Proposition \ref{prop:Edelta} below).

Estimating the dimensions of such limsup sets is crucial for understanding the geometric properties of fractals. Getting an upper bound for the Hausdorff dimension of these sets is often much easier to establish than the corresponding lower bounds. 
A common approach to obtain lower bounds relies on the Mass Transference Principle (MTP), a powerful result that establishes a link between measure-theoretic and geometric properties. 
\begin{theorem}[Mass Transference Principle]\cite{Beresn:06}\label{MTP}
Let $X$ be a compact set in $\R^d$ and assume that there exist $s\in [0, d]$ and $a,b,r_0>0$ such that
\begin{equation}\label{eqGMT}
a r^s \leq \mathcal{H}^{s}(B \cap X) \leq b r^s
\end{equation}
for any ball $B$ of center $x \in X$ and of radius $r \leq r_0$. Let $\delta >0$. Given a ball $B = B(x,r)$ with center in $X$, we set
$
B^{\delta} = B \left(x, r^{\delta} \right).
$
Assume that $(B_n)_{n \in \N}$ is a sequence of balls with center in $X$ and radius $r_n$  such that the sequence $(r_n)_{n \in \N}$ converges to 0. If 
\[
\mathcal{H}^s \left(  X \cap \limsup_{n \rightarrow + \infty} B_n \right) = \mathcal{H}^s( X),
\]
then
\[
\mathcal{H}^{\frac{s}{\delta}} \left(X \cap \limsup_{n \rightarrow + \infty} B_n^\delta \right) = \mathcal{H}^{\frac{s}{\delta}}(X).
\]
\end{theorem}
Here, $\mathcal{H}^s$ denotes the Hausdorff measure of dimension $s$. 
Historically, the first result of this type was obtained by Jaffard in \cite{Jaffard:00b} in the context of multifractal analysis of lacunary wavelet series (LWS) on $[0,1]$, and states that if 
$
\mathcal{L}(\limsup_{n\to + \infty} B_n) = 1,$
then
$$
\dim_{\mathcal{H}}(\limsup_{n\to + \infty} B_n^\delta) \geq \frac{1}{\delta},
$$
where $\dim_{\mathcal{H}}$ is the  Hausdorff dimension, see Definition \ref{def:hausdorff}. 
The key argument of both proofs is to construct a generalized Cantor set included in the limsup set of the contracted balls and, simultaneously, a probability measure supported by this Cantor set with prescribed scaling properties. The MTP and its extensions has proven effective in a variety of settings, making it a key tool for studying the dimensions of fractal sets and their associated structures, see e.g. \cite{Beresn:06,BaSe07,Koivusalo22,Daviaud:24}. 

Regarding the multifractal analysis of functions, despite the significant advances in the understanding of the multifractal analysis, both in theory and practice, important challenges remain. Let us mention, among others, the refinement of large deviation methods, the multivariate analysis of functions, and the multifractal analysis of signals with lacunar supports.  The analysis of functions or stochastic processes defined on fractal sets is a growing area of research, particularly in fields where data is naturally supported on fragmented or irregular structures. This is especially relevant in geographical studies, where multifractal analysis of parcel-based data is essential (\cite{LenRoux:22a, LenRoux:22b}). Recent advances have focused on defining and studying random fields on fractal supports, such as Sierpinski carpets \cite{BaudoinLacaux:20}, which serve as prototypical examples of fractal domains. These works extend classical models of random fields to accommodate the intricate geometry and self-similarity of fractals. The multivariate analysis of functions makes also naturally emerge fractal supports since it aims to describe the intersection of iso-H\"older sets (\cite{JaffSeu:19,Seuret:24}). 

Building on this context, in the recent work done in \cite{Jaffard:00b},  we have investigated the overestimation of the Hausdorff dimension of iso-Hölder sets by classical large deviation methods \cite{EV23}. Focusing on functions represented in a wavelet basis, we have proposed a method based on ``lacunarized expansions'' of the function. Specifically, we have shown that for so-called ``$\alpha$-sparse wavelet series'', large deviation methods for the lacunarized wavelet series provides a criterion for detecting such overestimations. We refer to Section 2 for the definition of wavelet bases and $\alpha$-sparse wavelet  series. This approach offers a computationally feasible way to refine the multifractal analysis of signals and images. 

Based on this work, the present paper further explores the multifractal properties of wavelet series defined over fractal subsets, where the upper box-counting dimension matches the Hausdorff dimension.

Let us be more precise in the description of the model under study within this paper. Let $\I \subset [0,1]$ be a compact set satisfying 
\begin{equation}\label{eq:egaldim}
    \dim_{\mathcal{H}} (\I) = \dim_{u-box} (\I)>0,
\end{equation}
let $\alpha>0$ and let $\eta \in (0, \dim_{\mathcal{H}} (\I))$. Given a wavelet $\psi$, the \emph{lacunary wavelet series $f$ on $\I$ of parameters $(\alpha, \eta)$} is the random series defined by 
$$ 
f = \sum_{j \geq 0} \sum_{k \in \I_j} 2^{-\alpha j} \xi_{j,k}\psi(2^j \cdot -k) 
$$ 
where 
\begin{equation}\label{eq:Ij} \I_j= \{ k \in \{0, \dots, 2^j-1\} : [k2^{-j}, (k+1)2^{-j}) \cap \, \I \neq \emptyset\} 
\end{equation} 
and where $(\xi_{j,k})_{j,k}$ is a sequence of independent Bernoulli random variables with parameter $2^{(\eta-\dim_{\mathcal{H}} (\I))j}$. Basic definitions regarding wavelets are provided in Section \ref{rappel}. As mentioned above, the multifractal analysis of the basic case corresponding to $\I=[0,1]$ was studied by Jaffard \cite{Jaffard:00b}. Despite its very simple structure, this wavelet series already reveals a very interesting multifractal structure. When a fractal set $\I$ satisfy the assumptions of more general versions of the MTP, the multifractal study of LWS defined on $\I$ become feasible, using similar arguments to those of Jaffard for the case $[0,1]$. It is the case for Cantor subsets and more generally for some attractors of Iterated Function Systems (see \cite{EV23} for the study of LWS on Cantor sets).

In contexts where classical Mass Transference Principles cannot be applied, alternative strategies are required to obtain lower bounds for the dimension of limsup sets. To address this issue and enable multifractal analysis of lacunary wavelet series on fractals $\mathcal{I}$ -- which satisfy the equality between their Hausdorff and box-counting dimensions but, for example, lack separation properties -- we use another strategy. 
Before considering the limsup of balls, we perform  some ``preconditionning'' of the fractal set. This consists in constructing compact subsets $\mathcal{K}$, which we call quasi-Cantor sets, within $\mathcal{I}$. These subsets exhibit controlled self-similarity on specific dyadic scales and have a dimension that can be made arbitrarily close to that of $\I$. This controlled self-similarity facilitates accurate dimension estimation of limsup sets of balls centered at random points in $\mathcal{K}$. This approach avoids the limitations of classical MTPs and allows to get useful geometric information even in cases where standard assumptions fail.  Moreover, it not only generalizes classical results on lacunary series but also offers deeper understanding of the fractal geometry of the set $\mathcal{I}$.

Let us state our mains results. The first one provides the announced construction of quasi-Cantor subsets of $\I$, by controlling the duplication rate up to some $\varepsilon > 0$, from a scale $J$ to scales $(1+b)^\ell J$ for any $\ell\geq 0$. 

\begin{theorem}\label{theo:Keps}
Let $\I \subset [0,1]$ be a compact set satisfying $$\dim_{\mathcal{H}} (\I)= \dim_{u-box} (\I)>0$$  and denote by $\I_j$ the set of dyadic intervals of size $2^{-j}$ that intersect $\I$. Let $\varepsilon>0$ and $ b \in (0,1)$. There exists a set $\K_{\eps}(b) \subset \I$ and $J, \ell_0 \in \N$ such that
\begin{enumerate}
\item $ \dim_{\mathcal{H}}  (\I) -\eps \le  \dim_{\mathcal{H}}  (\K_{\eps}(b) )\le  \dim_{\mathcal{H}}  (\I)$
\item for all $j=\lfloor (1+b)^\ell J \rfloor$ with $\ell \ge \ell_0$,  one has
$$2^{j( \dim_{\mathcal{H}} (\I)-\eps)} \le \# \{ \lambda \in \I_j : \, \lambda \cap \K_{\eps}(b) \neq \emptyset \} \le 2^{j( \dim_{\mathcal{H}} (\I)+\eps)} ,$$ 
\item for each dyadic interval $\lambda \in \I_j$ such that $\lambda \cap \K_{\eps}(b)  \neq \emptyset$, with  $j=\lfloor(1+b)^\ell J\rfloor$ and $\ell \ge \ell_0$,  and for any $\ell' \ge 1$, one has
$$
 \# \{ \lambda' \in \I_{\lfloor(1+b)^{\ell+\ell'}J\rfloor}: \, \lambda' \subset \lambda \, ,  \, \lambda' \cap \K_{\eps}(b)  \neq \emptyset \}\ge 2^{((1+b)^{\ell'}-1)(1+b)^\ell J(  \dim_{\mathcal{H}} (I)-\frac{5\eps}{b})} 
$$
\end{enumerate}
Such a set is called an $(\eps,b)$-quasi-Cantor subset of $\I$.
\end{theorem}

The previous theorem is the main tool for studying the multifractal properties of lacunary wavelet series on $\mathcal{I}$, as stated in our second main result.

\begin{theorem}\label{thm:spectreLWS}
Let $\I \subset [0,1]$ be a compact set satisfying $\dim_{\mathcal{H}} (\I)= \dim_{u-box} (\I)>0$,  let $\alpha>0$ and let   $\eta \in (0, \dim_{\mathcal{H}} (\I))$.  Assume that the regularity $r(\psi)$  of the wavelet is  greater than $\frac{\dim_{\mathcal{H}} (\I)}{\eta} +1$. 
 If $f$ is the lacunary wavelet series on $\I$ of parameters $(\alpha, \eta)$, then
\begin{equation}\label{eq:spectre}
  \dim_{\mathcal{H}} \big\{ x \in [0,1] : h_f(x) \leq h \} = \frac{\eta}{\alpha}h\, , \quad \forall h \in [\alpha, \alpha \frac{\dim_{\mathcal{H}}(\I)}{\eta}]
\end{equation}
where  $h_f(x)$ denotes the H\"older exponent of $f$ at $x$, see Definition \ref{def:holder}. 
\end{theorem}

Fractal sets which fail to satisfy the assumption \eqref{eqGMT} - known as the Ahlfors regularity - but satisfy \eqref{eq:egaldim} are numerous in classical fractal geometry. To conclude this introduction, let us consider in detail the notable examples of attractors of IFS. Let $S=\{f_1, \dots, f_n \}$ denote a finite set of contracting similitudes on $\mathbb{R}$ with respective Lipschitz constants $r_1, \dots, r_n$. Then, there exists a unique compact and non-empty set $\mathcal{I}$ such that
    $$
    \mathcal{I} = \bigcup_{j=1}^n f_j(\mathcal{I}),
    $$
which is called the attractor of the IFS $S$. The attractor $\I$ easily satisfies the relation \eqref{eq:egaldim}. To $S$, one can also associate its similarity dimension $\dim(S)$, defined as the unique $\alpha$ solution of  the equation 
$$
\sum_{j=1}^n r_j^\alpha = 1.
$$
When the smaller copies of $\I$ are sufficiently well-separated, the fine-scale structure of the set becomes relatively straightforward to analyze, and, in particular, the Hausdorff dimension can be determined explicitly using the properties of the defining similitudes. Specifically, if the IFS $S$ satisfies the open set condition (OSC), the Hausdorff dimension of $\mathcal{I}$ is equal to $\dim(S)$. The OSC requires the existence of a non-empty open set $V$ such that the images $f_j(V)$, for $j \in {1, \dots, n}$, are mutually disjoint and entirely contained within $V$. The following result, established in \cite{Schief}, demonstrates that the OSC is also the key separation condition for ensuring that self-similar fractals possess a positive Hausdorff measure.
 
\begin{proposition}\cite{Schief}
Let $\mathcal{I}$ be the attractor of an IFS $S$ such that $\dim (S )= \dim_{\mathcal{H}}(\I)$. Then, $S$ satisfies the  OSC  if and only if 
    $$
    \mathcal{H}^{\dim_{\mathcal{H}}(\I)} (\mathcal{I})>0.
    $$
\end{proposition}

The previous result shows that for any IFS that does not satisfy the OSC but for which the expected equality  $\dim_{\mathcal{H}}(\I)=\dim(S)$ holds, the assumption  \eqref{eqGMT} is not fulfilled. Note that the OSC is known to be very restrictive, see e.g. \cite{Li15}. To give a tangible example, the IFS $S=\{x/3,(x+1)/3, (x+u)/3\}$ does not satisfy the OSC if $u$ is irrational \cite{Kenyon95}. 

The paper is organized as follows. Section \ref{rappel} provides a summary of key definitions related to dimensions, wavelets, Hölder regularity, and multifractal analysis. Section \ref{sec:goodrate} focuses on results and proofs from \cite{EV23}, which estimate the number of dyadic intervals that replicate at the expected rate -- relatively to the fractal set’s dimension -- at the subsequent scale. These estimates serve as the foundation for constructing the fractal set  $\K_{\eps}(b)$ in Section \ref{sec:Keps}, where Theorem \ref{theo:Keps} is proved. Finally, Section \ref{sec:LWS} is dedicated to the proof of Theorem \ref{thm:spectreLWS}.

\section{Basic definitions and classical results on multifractal analysis}\label{rappel}
\subsection{Fractal dimensions}

In this subsection, we introduce the fundamental concepts of fractal dimensions needed in this paper.  We refer to \cite{Falconer:86}, for example,  for more details.

\begin{definition}[Hausdorff measure and Hausdorff and upper-box dimensions]\label{def:hausdorff}
Let $\I \subset \R$ and $\delta >0$. For $ s \in [0,1]$, set
$$
\mathcal{H}^s_{\delta} (\I)= \inf \left\{  \sum_{i\in \N}
  \diam(A_i)^{s} :  \I \subset \bigcup_{i\in \N} A_i\,\,  \text{\rm and
  }\diam(A_i) <\delta \,\, \forall i \in \N\right\}.
$$
The $\delta$-dimensional Hausdorff measure of $\I$ is $\mathcal{H}^s(\I) =
\lim_{\delta \to 0} \mathcal{H}^s_{\delta}(\I)$ and the Hausdorff dimension of $\I$ is given by
$$
\dim_{\mathcal{H}}(\I) = \inf \{ s \ge 0 : \, \mathcal{H}^s(E) = 0 \}
= \sup \{ s \ge 0 :  \, \mathcal{H}^s(\I)= +\infty\}.
$$

\end{definition}
We take the usual convention that $\dim_{\mathcal{H}}(\emptyset)= -\infty$. 

\begin{definition}
    Let $\I \subset \R$. The upper-box dimension of $\I$ is given by
$$
\dim_{u-box}(\I) := \limsup_{\varepsilon \to 0} \frac{\log N(\varepsilon)}{-\log(\varepsilon)}
$$
where $N(\varepsilon)$ is the smallest number of intervals of length $\varepsilon$ which can cover $\I$. 
\end{definition}

\noindent The following relationship between the two dimensions always holds.

\begin{lemma}
For any set $\I$, one has
\begin{equation}
\dim_{\mathcal{H}}(\I) \le \dim_{u-box}(\I).
\end{equation}
\end{lemma}
\subsection{Multifractal analysis and computation via wavelets}
The aim of multifractal analysis is to describe the local change of smoothness of a function (or a measure). For functions, the smoothness is characterized by the pointwise H\"older  exponent, defined as follows.
\begin{definition}\label{def:holder}
Let $x_0 \in \R$ and $h>0$. A locally bounded function $f : \R \to \R$ belongs to $\C^{h}(x_0)$ if there exists $C>0$ and a polynomial $P_{x_0}$ with $\deg P_{x_{0}} < \lfloor h \rfloor$  such that 
$$
\vert f(x)-P_{x_0}(x) \vert \le C \vert x -x_0 \vert^{h}
$$
on a neighborhood of $x_0$. The \emph{pointwise H\"older exponent} of $f$ at $x_0$ is 
$$
h_f(x_0) = \sup \{ h\geq 0 : \, f \in \C^{h}(x_0)\}.
$$
\end{definition}
In many situations and in particular for stochastic processes for which it can change at each realization, we are not able to compute the pointwise exponent at each point. However, a relevant information is given by the dimension of the iso-H\"older set for any level $h \in [0, + \infty]$, which consists of all points  where the function $f$ has a H\"older exponent exactly equal to $h$. This information is given by the multifractal spectrum.
\begin{definition}\label{def:spectrum}
The \emph{multifractal spectrum} $\mathcal{D}_{f}$ of a locally bounded function $f$ is the function 
$$
 {\cal{D}}_f : h \in  [0, + \infty] \mapsto \dim_{\mathcal{H}}\{x_0 \in \R : \, h_f(x_0) = h\}.
$$
The \emph{increasing multifracal spectrum} of $f$ is the function 
$$
 {\cal{D}}_{f,\le} : h \in  [0, + \infty] \mapsto \dim_{\mathcal{H}}\{x_0 \in \R : \, h_f(x_0) \leq h\}.
$$
\end{definition}

The pointwise regularity of a locally bounded function $f$ can be characterized with the help of a wavelet analysis of the function. Before recalling this classical result, let us give the definition of a wavelet basis.  We refer to \cite{Daubechies:92,Meyer:95,Mallat:99} e.g. for more information about wavelets.

An orthonormal wavelet basis on $\R$ is given by two functions
$\varphi$ and  $\psi$ with the property that
the family
$$
\{ \varphi(\cdot -k) :\, k \in \Z\} \cup \{ 2^{\frac{j}{2}} \psi(2^j
\cdot -k): \, j \in \N, k \in \Z \}
$$
forms an orthonormal basis of $L^2(\R)$. Therefore, for all $f \in L^2(\R)$, we have the following decomposition 
$$
f= \sum_{k \in \Z} C_k \varphi(\cdot -k) +  \sum_{j \in \N} \sum_{k \in \Z} c_{j,k} \psi(2^j\cdot -k)
$$
where the wavelet coefficients of $f$ are given by
$$
C_k = \int_{\R} f(x) \varphi(x-k) dx
$$
and
$$
c_{j,k} = 2^{j} \int_{\R} f(x) \psi(2^jx -k)dx.
$$
Note that we do not use the $L^2$-normalisation to avoid a rescaling
in the definition of the wavelet leaders, see Definition
\ref{def:leader} below. Note also that the definition of the wavelet
coefficients makes sense even if $f$ does not belong to $L^{2}(\R)$. 

We say in addition that the wavelet basis is $r$-smooth if $\varphi$ and
$\psi$ have partial derivatives up to order $r$ and if these partial
derivatives have fast decay. In this case, the wavelet $\psi$ has a
corresponding number of vanishing moments. Constructions of wavelet bases with arbitrarily large $r$ have been proposed in  \cite{Daubechies:92}. One can also chose a wavelet basis with infinitely smooth functions \cite{Meyer:95}.

Usually, the following compact
notations using dyadic intervals are used for indexing wavelets. If
$\lambda=\lambda_{j,k}= [k2^{-j}, (k+1)2^{-j}[$, we write
$c_{\lambda}=c_{j,k}$ and $\psi_{\lambda}= \psi_{j,k}=
\psi(2^j\cdot -k)$.  These notations are
justified by the fact that the wavelet $\psi_{\lambda}$ is essentially localized on the cube $\lambda$ in the following way : if the wavelets are compactly supported then
$$
\exists C>0 \text{ such that } \forall \lambda \quad \supp ( \psi_{\lambda}) \subset C\,  \lambda
$$ 
where $C\, \lambda$ denotes the interval of same center as $\lambda$
and $C$ times wider.

\begin{definition}\label{def:leader} Let $\lambda$ be a dyadic cube and $3 \lambda$ the cube of same center and three times wider. If $f$ is a bounded function, the wavelet leader $d_{\lambda}$ of $f$ is given by
$$
d_{\lambda} = \sup_{\lambda' \subset 3 \lambda} \vert c_{\lambda'} \vert.
$$
\end{definition}

The pointwise H\"older regularity of the function $f$ at a point $x_0$
can be determined using the wavelet leaders \cite{Jaffard:04b}. Let
$x_0 \in \R$, the notation $\lambda_j(x_0)$ refers to the dyadic cube of width $2^{-j}$ which contains $x_0$ and 
$$
d_j(x_0) = d_{\lambda_j(x_0)}= \sup_{\lambda' \subset 3 \lambda_j(x_0)} \vert c_{\lambda'} \vert.
$$

\begin{theorem}\cite{Jaffard:04b}\label{thm:waveletcharact}
Let $h>0$ and $x_0 \in \R$. Assume that $f$ is a bounded function and that the wavelet basis is $r$-smooth with $r > [h]+1$. 
\begin{enumerate}
\item Suppose  $f$ belongs to $\C^{h}(x_0)$. Then there exists $C>0$ such that
\begin{equation}\label{leaderholder}
\forall j \ge 0, \quad d_j(x_0) \le C2^{-hj}.
\end{equation}
\item Conversely, suppose (\ref{leaderholder}) holds and that $f$
  belongs to $\C^{\varepsilon}(\R)$ for some $\varepsilon >0$. Then $f$ belongs to $C^{h'}(x_0)$ for all $h'<h$. In particular, $h_f(x_0) \ge h$.
\item  Suppose $f \in \C^{\varepsilon}(\R)$. Then $$h_f(x_0) = \liminf_{j\to +\infty} \frac{\log (d_j(x_0))}{\log (2^{-j})}.$$
\end{enumerate}
\end{theorem}

In what follows, we will thus always assume that the wavelet $\psi$ is
$r$-smooth with $r = r(\psi)$ large enough.

Consider now an $\alpha$-sparse wavelet series. By this, we mean a wavelet series $f$ 
$$
f = \sum_{j \in \N} \sum_{0 \le k \le 2^{j}-1} c_{j,k} \psi_{j,k} 
$$
with
$$
c_{j,k} = \begin{cases} 2^{-\alpha j} &\text{ if } k \in F_j\\ 0 &\text{ otherwise } \end{cases}
$$
where $F_j \subset \{0,\dots,  2^{j}-1\}$.
As a direct consequence of Theorem \cite{Jaffard:04b}, we have the following characterization of sets of points whose regularity is smaller than $h$. The proof is straighforward,  is provided in \cite{EV23} for the sake of completeness.

\begin{proposition}\label{prop:Edelta}
Let $f$ be a $\alpha$-sparse signal. For every $\delta
\in (0,1]$, we define the set 
\[
E_{\delta}(f):= \limsup_{j \rightarrow + \infty}\bigcup_{k \in F_{j}
}B\left( k 2^{-j}, 2^{-\delta (1-\varepsilon_{j})j} \right)
\]
where $\varepsilon_{j}$ has a null limit as $j \to \infty$. One has
\begin{enumerate}
\item  If $x \in E_{\delta}(f)$, then $h_{f}(x) \leq \frac{\alpha}{\delta}$.
\item  If $x \notin E_{\delta}(f)$, then $h_{f}(x) \geq \frac{\alpha}{\delta}$.
\end{enumerate}
\end{proposition}

As a consequence, the study of the increasing multifractal spectrum of a lacunary wavelet series, which is a random $\alpha$-sparse signal, can be obtained by the computation of the Hausdorff dimension of some limsup of random balls.

When in addition, one is able to construct a positive measure on the sets $E_{\delta}(f)$, it makes possible to obtain the dimension of the iso-h\"older sets and hence compute the spectrum of singularities. It has been done for lacunary wavelet series on $[0,1]$ by Jaffard in \cite{Jaffard:00}.
\begin{theorem}\cite{Jaffard:00}
    Let $f$ be a LWS on $[0,1]$ of parameters $\alpha>0$ and $\eta \in(0,1)$. Almost surely, one has 
    $$
    \mathcal{D}_{f}(h) = \begin{cases}  \frac{\eta}{\alpha}h & \text{ if } h \in \big[\alpha, \frac{\alpha}{\eta}\big], \\[1ex]
    - \infty & \text{ otherwise. }\end{cases}
    $$
\end{theorem}

Using the more general MTP of \cite{Beresn:06}, the result can be easily extended to symmetric Cantor sets, see \cite{EV23}. We denote by $\C(r)$ the symmetric Cantor set with a dissection ratio  $r<\frac{1}{2}$, having a Hausdorff dimension given by
$$\dim_{\cal H}\, \C(r) = - \frac{1 }{\log_2 r}.$$

\begin{proposition}\label{thm:LWS}\cite{EV23}
 Let $f$ be a LWS on  $\C(r)$ of parameters $\alpha>0$ and $\eta \in(0,\dim_{\cal H} \mathcal{C}(r))$.   Almost surely, one has
  $$
  \mathcal{D}_{f}(h)  =\begin{cases}
    \frac{ \eta}{\alpha}h & \text{ if } h \in [\alpha, \frac{\alpha\dim_{\cal H}\, \C(r) }{\eta}],
    \\[1.5ex]
    1 & \text{ if } h=r(\psi), \\[1ex]
  - \infty & \text{ otherwise. }\end{cases}
  $$
\end{proposition}

\section{Sets with good rate of duplication }\label{sec:goodrate}

In this section, we recall two results from \cite{EV23} that provide an estimation of the number of dyadic intervals that duplicate at the expected rate at the next scale, with respect to the dimension of the fractal set. Since these results will serve as the foundation for the iterative procedure used to construct the quasi-Cantor set $\K_{\eps}(b)$ in Theorem \ref{theo:Keps} in the next section, we have included the proofs of these results in the Appendix. First, let us introduce some notations.

\medskip

\noindent {\bf Notations :} 
\begin{itemize}
    \item In what follows, we assume that $\I$ is a compact  subset of $[0,1]$ such that
$$
 \dim_{\mathcal{H}}  \I = \dim_{\text{u-box}} \I >0. 
$$We denote this common dimension by $H$. 
    \item Let
$$
\I_j = \{ \lambda= \left[ 2^{-j}k, 2^{-j}(k+1) \right] :  {\lambda} \cap \I \neq \emptyset \}.
$$
be the collection of dyadic intervals of size $2^{-j}$ that intersect the fractal set $\I$.
\item For a collection $U = U(J)$ of dyadic intervals at a scale $J$, we denote the collection of all their children at the scale $(1+\beta)J$ by
$$
C_{\beta}U := \{ \lambda' \in \I_{(1+\beta)J} : \lambda' \subset \lambda \text{ with } \lambda \in U \}. 
$$
\item We also denote by $\widetilde{U}$ the points $x \in [0,1]$ in the union of the dyadic intervals which belong to $U$. 
\end{itemize}

\begin{remark}\label{rem:partieentiere} Obviously, the scale $(1+\beta)J$ has no reason to be an integer and so a scale. We should consider instead its entire part. This is what is done in the following of the current section but it does not play a crucial role. In order to ease the reading of Section \ref{sec:Keps} with the construction of the set $\Ke(b)$, we will forget the entire parts of the scales in the remaining part of the paper.
\end{remark}

Let us fix $\varepsilon >0$.

\begin{lemma}\label{lem:estimI_j}
For every $\varepsilon>0$, there exists $J \in \N$ such that for all $j \ge J$, one has
$$
2^{j(H-\varepsilon)} \le \# \I_j \le 2^{j(H+\varepsilon)}.
$$
\end{lemma}

We now introduce different collections of intervals $\lambda$ of scale $J$ by considering the number of their children at scale $\lfloor (1+\beta)J \rfloor$  for a fixed $\beta>0$ and up to some $\varepsilon>0$.

\begin{definition} Let $\varepsilon >0$, $\beta >0$ and $j \in \N$. We set $m=\max(1,\beta)$.
\begin{enumerate}
\item We say that $\lambda \in \I_j$ has a {\it{slow duplication rate}} if
\begin{equation}\label{Eq:SDR}
 \# \{ \lambda' \subset \lambda : \, \lambda' \in \I_{\lfloor (1+ \beta) j \rfloor} \}\le 2^{j (\beta H- 4 m\varepsilon)} 
\end{equation}
and we denote by $SD(j,\beta,\eps)$ the collection of such intervals.

\item we say that $\lambda \in \I_j$ has a {\it{normal duplication rate}} if
\begin{equation}\label{Eq:NDR}
2^{j(\beta H -4 m\varepsilon)} \le  \# \{ \lambda' \subset \lambda \cap \I_{\lfloor (1+\beta)j \rfloor} \} \le  2^{j(\beta H +4 m \varepsilon)}.
\end{equation}
and we denote by $ND(j,\beta,\eps)$ the collection of such intervals.
\item We say that $\lambda \in \I_j$ has a {\it{fast duplication rate}} if
\begin{equation}\label{Eq:FDR}
 \# \{ \lambda' \subset \lambda : \, \lambda' \in \I_{\lfloor (1+ \beta) j \rfloor} \}\ge 2^{j (\beta H + 4 m  \varepsilon)} 
\end{equation}
and we denote by $FD(j,\beta,\eps)$ the collection of such intervals.
\end{enumerate}
\end{definition}

In \cite{EV23}, we obtained the following lower and upper bound of the cardinalities of each sets.

\begin{proposition}\cite{EV23}\label{prop:CardSNF}
Let $\varepsilon>0$, $\beta>0$ and $m=\max(1,\beta)$. There exists $J_0 \in \N$ such that for all $j \ge J_0$, one has
\begin{equation}\label{Eq:CardND}
2^{j(H-2\eps)} \le \# ND(j,\beta,\eps) \le 2^{j(H+\eps)}
\end{equation}
\begin{equation}\label{Eq:CardFD}
\# FD(j,\beta,\eps) \le 2^{j(H-2\eps)}
\end{equation}
and
\begin{equation}\label{Eq:CardSD}
\# C_{\beta}SD(j,\beta,\eps) \le 2^{j(1+\beta)(H-3m\eps)}
\end{equation}
\end{proposition}

\section{Quasi-Cantor subsets of $\I$}\label{sec:Keps}

The aim of this section is to prove Theorem \ref{theo:Keps}. The results of Section \ref{sec:goodrate} do not establish the existence of sets that duplicate at the ``good rate'' at all scales, as is the case for a Cantor set. Indeed, the children of such sets may exhibit either slow or fast duplication. Here, we propose an iterative procedure to retain only the sets that exhibit an ``infinite good rate'' of duplication.

We now consider a positive number $b \in (0,1)$ and a sufficiently small $\varepsilon > 0$ such that  
$$  
bH - 5\varepsilon > 0.  
$$  
We introduce sub-collections of $ND(J, b, \varepsilon)$, $J \in \N$, in order to accurately analyze if the rate of duplication is preserved through the scales for an expected number of dyadic intervals. Note that $\max(1,b)=1$.
We fix $J_0$ large enough such that the estimations of the cardinality of $FD(J,b,\eps)$, $ND(J,b,\eps)$ and $C_{b}SD(J,b,\eps)$ given in Proposition \ref{prop:CardSNF} hold for every $J \geq J_0$. 
We  now define a succession of nested sets. As mentioned in Remark \ref{rem:partieentiere}, we will deliberately omit the minor adjustments required to account for the fact that the scales $(1+b)^\ell J$ should be considered in their integer part. For all $J \ge J_0$, we set
$$
T_1 (J, b, \eps) = ND(J,b,\eps), 
$$
and
$$
T_2(J,b,\eps) := \{ \lambda \in T_1 (J, b, \eps) : \, \# \{ \lambda' \subset \lambda : \, \lambda' \in T_1((1+b)J,b,\eps)\} \ge 2^{J(bH-5\eps)} \}.
$$
Next, we define recursively for $\ell \geq 3$,
$$
T_\ell(J,b,\eps) :=\{ \lambda \in T_{k-1}(J,b,\eps) :  \, \# \{ \lambda' \subset \lambda :  \lambda' \in T_\ell((1+b)J,b,\eps)\} \ge 2^{J(bH-5  \eps)}\}.
$$
We also define the complementary sets
$$
U_\ell(J,b,\eps) = T_{\ell-1}(J,b,\eps) \setminus T_\ell(J,b,\eps)
$$
for all $\ell \geq 2 $.
The following lemma provides an estimate of the maximal number of intervals required to cover the complementary sets.

\begin{lemma}\label{lem:U1bis}
For every $J \geq J_0$, the set $\widetilde{U}_2(J,b,\eps) \cap \I$ is covered by the union of
\begin{enumerate}
    \item at most $\# U_2(J,b,\eps) \times 2^{J(bH-5\eps)}$ intervals of length $2^{-(1+b)J}$,
    \item intervals of $FD((1+b)J,b,\eps)$
    \item intervals of $C_{b}SD((1+b)J,b,\eps)$.
\end{enumerate}

\end{lemma}

\begin{proof}
Let us cover the set $\widetilde{U}_2(J,b,\eps)\cap \I$ by considering the set of children of $U_2(J,b,\eps)$ at the scale $(1+b)J$ and by decomposing this set following the rate of duplication of the children at the scale $(1+b)^2J$. We obtain
$$
C_{b}U_2(J,b,\eps) 
\subset FD((1+b)J,b,\eps) \cup SD((1+b)J,b,\eps) \cup \left( ND((1+b)J,b,\eps) \cap C_{b} U_2(J,b,\eps)\right).    
$$
The set $\widetilde{SD}((1+b)J,b,\eps)$ can be covered by its children at the next generation, that is by the union of the intervals of $C_{b} SD((1+b)J,b,\eps)$. Finally, by definition of $U_2(J,b,\eps)$, any dyadic interval $\lambda \in U_2(J,b,\eps)$ has less than $2^{J(bH-5\eps)}$ children in $ND((1+b)J,b,\eps)$. It follows that $\widetilde{ND}((1+b)J,b,\eps) \cap \widetilde{C_{b} U_2}(J,b,\eps)$ is covered by at most $\# U_2(J,b,\eps) \times 2^{J(bJ-5\eps)} $ intervals of length $2^{-(1+b)J}.$ 
\end{proof}

The previous lemma is the first step towards obtaining a covering for any $\widetilde{U}_\ell(J,b,\varepsilon)$ with $\ell \geq 2$.

\begin{lemma}\label{lem:denombrU} Let $\ell \ge 2$. For every $J \geq J_0$, the set  $\widetilde{U}_\ell(J,b,\eps)\cap \I$ is covered by the union of 
\begin{enumerate}
\item at most $\# U_\ell (J,b,\eps) \times 2^{J(bH-5 \eps)}$ intervals of length $2^{-(1+b)J}$,
\item at most, for each $p \in \{2, \dots, \ell-1\} $ and each $q \in \{p+1, \dots, \ell\}$, $$\# U_{q-p+1}((1+b)^{p-1},b,\eps) \times 2^{(1+b)^{p-1}J(bH-5\eps)}$$ intervals of length $2^{-(1+b)^pJ}$, 
\item for each $p\in \{2, \dots, \ell \}$, intervals of $FD((1+b)^{p-1}J,b,\eps)$ and of $C_{b} SD((1+b)^{p-1}J,b,\eps)$.
\end{enumerate}
\end{lemma}

\begin{proof}
Lemma \ref{lem:U1bis} gives the result for $\ell=2$ and for any $J \ge J_0$. Suppose that $\ell\geq 2$ and that the result is true for any $J \ge J_0$ and for all the sets $U_p(J,b,\eps)$ with $p \le \ell$. We want to obtain the result for $U_{\ell+1}(J,b,\eps)$. As done previously, we observe that 
$$ C_{b}U_{\ell+1}(J,b,\eps)
\subset FD((1+b)J,b,\eps) \cup SD((1+b)J,b,\eps) \cup \left(ND((1+b)J,b,\eps) \cap C_{b}U_{\ell+1}(J,b,\eps)\right).
$$
The intervals belonging to $SD((1+b)J,b,\eps) $ form a set that can be covered by the intervals of $C_{b}SD((1+b)J,b,\eps)$. 

 Moreover, since 
 $$ND((1+b)J,b,\eps) = T_{\ell}((1+b)J,b,\eps) \cup \bigcup_{p=3}^{\ell+1} U_{p-1}((1+b)J,b,\eps),$$ we get that
$$
 ND((1+b)J,b,\eps)  \cap   C_{b}U_3(J,b,\eps)  \subset  \left(T_\ell((1+b)J,b,\eps) \cap C_{b}U_{\ell+1}(J,b,\eps)\right) \cup \bigcup_{p=3}^{\ell+1} (U_{p-1}((1+b)J,b,\eps) .
$$
For all the terms appearing in the union,  we use our assumption at the scale $(1+b)J$ and with different generations $(1+b)^s, 2 \le s \le q+1$. 

Regarding the first term, if $\lambda$ belongs to $U_{\ell+1}(J,b,\eps)$, there are no more than $2^{J(bH-5 \eps)}$ of its children in $T_{\ell}((1+b)J,b,\eps)$. Hence the number of dyadic intervals in $T_{\ell}((1+b)J,b,\eps) \cap C_{b}U_{\ell+1}(J,b,\eps)$ is smaller than $\# U_{\ell+1}(J,b,\eps) \times 2^{J(bH-5 \eps)}$.
\end{proof}

Finally, we obtain the following proposition.

\begin{proposition}\label{prop:covS}
   For every $J \geq J_0$, we set 
    $$ 
    U_{\infty}(J,b,\eps):=\bigcup_{\ell=1}^{+ \infty} {U_\ell} (J,b,\eps).
    $$
    Then, the set $\widetilde{U_{\infty}}(J,b,\eps) \cap \I$ 
    is covered by  the union of 
    \begin{enumerate}
\item at most, for each $p \geq 1 $, 
$$\# ND((1+b)^{p-1}J,b,\eps) \times 2^{(1+b)^{p-1}J(bH-5\eps)}$$ intervals of length $2^{-(1+b)^pJ}$, 
\item for each $p\geq 2$, intervals of $FD((1+b)^{p-1}J,b,\eps)$,
\item for each $p \geq 2$, intervals of $C_{b} SD((1+b)^{p-1}J,b,\eps)$.
\end{enumerate}
\end{proposition}

\begin{proof}
Let us notice that the sets $U_\ell(J,b,\eps)$, $\ell \ge 2$ are disjoint and included in the sets $ND(J,b,\eps)$. Hence
$$
\sum_{\ell=2}^{+ \infty} \# U_\ell(J,b,\eps)  \le \#ND(J,b,\eps).
$$
Similarly, for every $p \geq 2$, one has
$$
\sum_{q=p+1}^{+ \infty} \# U_{q-p+1}((1+b)^{p-1}J,b,\eps) \leq \# ND((1+b)^{p-1}J, b, \eps)
$$
We obtain then directly the conclusion by applying Lemma \ref{lem:denombrU}. 
\end{proof}

The covering obtained in the previous result allows us to estimate the Hausdorff dimension of intersection of sets defined by  $$  
T_{\infty}(J, b, \varepsilon) := \bigcap_{\ell \geq 1} T_\ell(J, b, \varepsilon)  
$$  
for $J \geq J_0$. Before estimating the dimension, we first show that this set (which will be proven to be non-empty in Theorem \ref{theo:Vinf}) possesses a nice (almost) self-similar structure.

\begin{proposition}\label{prop:bonnereproduction}
 Fix $\ell \geq 1$. Any dyadic interval $\lambda \in T_{\infty}(J,b,\eps)$ satisfies
$$
\# \big\{ \lambda' \subset \lambda : \lambda' \in T_{\infty}((1+b)^\ell J,b,\eps) \big\}\ge 2^{((1+b)^\ell -1)J( H-5 \frac{\eps}{b})}. 
$$
\end{proposition}

\begin{proof}
We proceed by induction. For the case $\ell = 1$, assume that  
$$  
\# \big\{ \lambda' \subset \lambda : \lambda' \in T_{\infty}((1+b)J,b,\varepsilon) \big\} < 2^{J(b H - 5 \varepsilon)}.  
$$  
Since $\lambda$ has a finite number of other children (i.e., those not in $T_{\infty}((1+b)J,b,\varepsilon)$), we can define  
$$  
\ell_0 > \max_{\lambda' \subset \lambda, \lambda' \notin T_{\infty}((1+b)J,b,\varepsilon)} \max \{ \ell \in \mathbb{N} : \lambda' \in T_\ell((1+b)J,b,\varepsilon) \}.  
$$  
This choice of $\ell_0$ ensures that there are strictly less than  
$2^{J(b H - 5 \varepsilon)}$ children of $\lambda$ belonging to the set $T_{\ell_0}((1+b)J,b,\varepsilon)$. This contradicts the fact that $\lambda$ belongs to $T_{\ell_0+1}(J,b,\varepsilon)$.  

Now, assume that 
$$
\# \big\{ \lambda' \subset \lambda : \lambda' \in T_{\infty}((1+b)^\ell J,b,\eps) \big\}\ge 2^{((1+b)^\ell -1)J( H-5 \frac{\eps}{b})}
$$
and let us prove the result for $\ell+1$. 
By the first case, each child $\lambda'$ of an interval $ \lambda$ that belongs to $T_{\infty}((1+b)^\ell J,b,\eps)$ gives at least $2^{(1+b)^\ell J(b H-5 \eps)}$ children at scale $(1+b)^{\ell+1}J$ belonging to $T_{\infty}((1+b)^{\ell+1}J,b,\eps)$. It follows that
\begin{align*}
\# \big\{ \lambda' \subset \lambda : \lambda' \in T_{\infty}((1+b)^{\ell+1}J,b,\eps) \big\} 
& \geq 2^{(1+b)^\ell J(b H-5 \eps)} 2^{((1+b)^\ell -1)J( H-5 \frac{\eps}{b})}\\
& = 2^{((1+b)^{\ell+1}-1)J(H -5 \frac{\varepsilon}{b})}.
\end{align*}
\end{proof}

\begin{theorem}\label{theo:Vinf}
Assume that $J \ge J_0$. There exists $\ell_0 \in \N$ such that 
$$
\dim_{\mathcal{H}} \left(\bigcap_{\ell\ge \ell_0} \widetilde{T_{\infty}}((1+b)^{\ell} J,b,\eps)\cap \I \right) \geq H- \eps
$$
and for all $\ell \ge \ell_0$, 
$$ 2^{(1+b)^{\ell} J(H-\varepsilon)} \le \# T_{\infty}((1+b)^{\ell}J, b, \eps) \le 2^{(1+b)^{\ell} J(H+\varepsilon)}.$$ 
\end{theorem}

\begin{proof}
For all $\ell \geq 1$, let us start by setting 
$$R((1+b)^{\ell}J,b,\eps) =FD((1+b)^{\ell}J,b,\eps) \cup SD((1+b)^{\ell} J,b,\eps) 
\cup  U_{\infty}  ((1+b)^{\ell}J,b,\eps) .
$$
Note that $R((1+b)^{\ell} J, b, \varepsilon)$ is exactly the set of dyadic intervals of $\I_{(1+b)^{\ell} J}$ that do not belong to the set $T_{\infty}((1+b)^{\ell} J, b, \varepsilon)$.  

Using Proposition \ref{prop:covS}, we obtain a covering of the set  $\widetilde{R}((1+b)^{\ell}J,b,\eps) \cap \I$ 
by 
  \begin{enumerate}
\item at most, for each $p \geq 1 $, 
$$\# ND((1+b)^{\ell+p-1}J,b,\eps) \times 2^{(1+b)^{\ell+p-1}J(bH-5\eps)}$$ intervals of length $2^{-(1+b)^{\ell+p}J}$, 
\item for each $p\geq 1$, intervals of $FD((1+b)^{\ell+p-1}J,b,\eps)$,
\item for each $p \geq 1$, intervals of $C_{b} SD((1+b)^{\ell+p-1}J,b,\eps)$.
\end{enumerate}
This covering provides an upper bound of the $r$-dimensional Hausdorff measure of the set $\widetilde{R}((1+b)^{\ell}J,b,\eps) \cap \I$  for $r \leq 2^{-(1+b)^{\ell}J}$ as follows. One has
\begin{align}\label{eq:mesRJ}    
&\mathcal{H}^{H-\varepsilon}_{r}\left(\widetilde{R}((1+b)^{\ell}J,b,\eps) \cap \I \right) \nonumber \\
&\le \sum_{p \ge 1}  \# ND((1+b)^{\ell+p-1}J,b,\eps) \times 2^{(1+b)^{\ell+p-1} J(bH-5\varepsilon)} 2^{-(1+b)^{\ell+p} J (H-\varepsilon) }\nonumber  \\
& \quad + \sum_{p \geq 1}  \# FD((1+b)^{\ell+p-1}J,b,\eps) \times 2^{-(1+b)^{\ell+p-1} J (H-\varepsilon) } \nonumber \\
& \quad + \sum_{p \geq 1} \# C_{b} SD((1+b)^{\ell+p-1}J,b,\eps) \times 2^{-(1+b)^{\ell+p} J (H-\varepsilon) } 
\end{align}
Let us examine the three sums separately, using Proposition \ref{prop:CardSNF} for each estimation. First, we have
\begin{align}\label{eq:Hauss1}
&  \sum_{p \ge 1}  \# ND((1+b)^{\ell+p-1}J,b,\eps) \times 2^{(1+b)^{\ell+p-1} J(bH-5\varepsilon)} 2^{-(1+b)^{\ell+p} J (H-\varepsilon) }  \nonumber \\
& \leq \sum_{p \ge 1} 2^{(1+b)^{\ell+p-1}J(H+\varepsilon) } 2^{(1+b)^{\ell+p-1} J(bH-5\varepsilon)} 2^{-(1+b)^{\ell+p} J (H-\varepsilon) }  \nonumber \\ 
&  = \sum_{p \ge 1}  2^{(1+b)^{\ell+p-1}J\varepsilon (-3+b)}  \nonumber \\ 
& \leq \sum_{p \geq 1} 2^{-(1+b)^{\ell+p}J\varepsilon } 
\end{align}
since $b-3 \leq -(1+b) $. Secondly, we have
\begin{align}\label{eq:Hauss2}
   \sum_{p \geq 1}  \# FD((1+b)^{\ell+p-1}J,b,\eps) \times 2^{-(1+b)^{\ell+p-1} J (H-\varepsilon) } 
 & \leq \sum_{p \geq 1}  2^{(1+b)^{\ell+p-1}J (H-2\eps)} 2^{-(1+b)^{\ell+p-1} J (H-\varepsilon) } \nonumber \\
 & =  \sum_{p \geq 1}  2^{-(1+b)^{\ell+p-1}J \eps} 
\end{align}
and similarly
\begin{align}\label{eq:Hauss3}
  \sum_{p \geq 1}  \# C_{b} SD((1+b)^{\ell+p-1}J,b,\eps) \times 2^{-(1+b)^{\ell+p} J (H-\varepsilon) } 
 & \leq \sum_{p \geq 1} 2^{(1+b)^{\ell+p}J (H-3 \eps) } 2^{-(1+b)^{\ell+p} J (H-\varepsilon) }\nonumber \\
 & =  \sum_{p \geq 1}  2^{-(1+b)^{\ell+p}J 4 \eps } \nonumber\\
 & \leq  \sum_{p \geq 1} 2^{-(1+b)^{\ell+p}J\varepsilon } .
\end{align}
Putting \eqref{eq:mesRJ}, \eqref{eq:Hauss1}, \eqref{eq:Hauss2} and \eqref{eq:Hauss3} together, we obtain
\begin{align}
\mathcal{H}^{H-\varepsilon}_{r}\left(\widetilde{R}((1+b)^{\ell}J,b,\eps) \cap \I \right) 
    & \leq 2 \sum_{p \geq 1} 2^{-(1+b)^{\ell+p}J\varepsilon } +  \sum_{p \geq 1}  2^{-(1+b)^{\ell+p-1}J \eps} \nonumber \\
    & \leq 3  \sum_{p \geq 1}  2^{-(1+b)^{\ell+p-1}J \eps}  \nonumber \\
    & \leq C 2^{-(1+b)^{\ell} J \eps}
\end{align}
for some constant $C>0$. 

Now,  let us fix $\ell_0 \geq 1$. The previous upper bound implies that for every $r< 2^{-(1+b)^{\ell_0} J}$, one has 
$$
\mathcal{H}^{H-\varepsilon}_{r} \left(\bigcup_{\ell \ge \ell_0}\widetilde{R}((1+b)^{\ell} J,b,\eps)\cap \I \right) 
 \le C  \sum_{\ell \ge \ell_0} 2^{-(1+b)^{\ell}J \eps}
 \le  C 2^{-(1+b)^{\ell_0} J \varepsilon}
$$
which in turn implies that 
$$
\dim_{\mathcal{H}}\left( \bigcap_{\ell_0 \ge 1} \bigcup_{\ell \ge \ell_0} \widetilde{R}((1+b)^{\ell} J,b,\eps)\cap \I \right) \le  H-\varepsilon.
$$
Recalling that $\dim_{\mathcal{H}}\I = H$ and noting that
$$
\I = \left( \bigcup_{\ell_0 \ge 1} \bigcap_{\ell \ge \ell_0} \widetilde{T_{\infty}}((1+b)^{\ell} J, b,\eps)\cap \I 
\right) \cup\left( \bigcap_{\ell_0 \ge 1} \bigcup_{\ell \ge \ell_0} \widetilde{R}((1+b)^{\ell} J ,\eps)\cap \I \right),
$$
we obtain
$$
\dim_{\mathcal{H}}\left( \bigcup_{\ell_0 \ge 1} \bigcap_{\ell \ge \ell_0} \widetilde{T_{\infty}}((1+b)^{\ell} J, b,\eps)\cap \I  \right) =  H. 
$$
In particular, there is $\ell_0 \in \N$ such that 
$$
\dim_{\mathcal{H}}\left( \bigcap_{\ell \ge \ell_0} \widetilde{T_{\infty}}((1+b)^{\ell} J, b,\eps)\cap \I  \right) \ge H-\varepsilon. 
$$
It gives the conclusion of the first part of the Theorem.

\medskip

For the second part of the Theorem, note first that  one has
$$T_{\infty}((1+b)^{\ell} J,b, \eps) \subset ND((1+b)^{\ell} J, b, \varepsilon)$$
by definition. Hence,  Proposition \ref{prop:CardSNF} directly gives $$
 \# T_{\infty}((1+b)^{\ell} J, b, \eps) \le 2^{(1+b)^{\ell} J(H+\eps)}.
$$
The lower bound is obtained by contradiction using the first part of the proof :  if there were infinitely many $q \ge q_0$ such that
$$
\# T_{\infty}((1+b)^{\ell} J, b, \eps)  < 2^{(1+b)^{\ell} J(H-\varepsilon)} ,$$
it would provide a sequence of covering of $\I$ implying that $\dim_{\mathcal{H}}(\I) \leq H-\varepsilon$.
\end{proof}

For a fixed $J \geq J_0$, if $\ell_0$ is the number given by Theorem \ref{theo:Vinf}, we set
\begin{equation}\label{eq:K_eps}
    \K_{\varepsilon}(b) = \bigcap_{\ell \geq \ell_0 }\widetilde{T_{\infty}}((1+b)^{\ell} J,b,\eps)\cap \I 
\end{equation}

It remains to prove the third point of Theorem \ref{theo:Keps}, which is a consequence of Proposition \ref{prop:bonnereproduction}. Indeed, let $\lambda \subset \I_j$ be such that $\lambda \cap \Ke(b) \neq \emptyset$ with $j=(1+b)^{\ell} J$ and $\ell \ge \ell_0$. Such a $\lambda$ belongs by definition to $T_{\infty}((1+b)^{\ell}J,b,\eps)$ and it follows by Proposition \ref{prop:bonnereproduction} that
$$
\# \{ \lambda' \subset \lambda\, : \, \lambda' \in T_{\infty}((1+b)^{\ell+\ell'}J,b,\eps) \} \ge 2^{(1+b)^{\ell'}-1)(1+b)^\ell J(H-5\frac{\eps}{b})}
$$
for any $\ell' \ge 1$.  
Such a $\lambda'$ intersects obviously the set $\K_{\eps}(b)$ since 
it contains children in $T_{\infty}((1+b)^{\ell''}J,b,\eps)$ at any scale $\ell'' \ge \ell+\ell'+1$.

\section{LWS on the compact fractal set $\I$}\label{sec:LWS}

This section applies the results obtained in the previous section to the study of the regularity of lacunary wavelet series ) supported on the compact fractal set $\I$. Specifically, we determine their increasing multifractal spectrum by establishing a Transference Mass Principle for the limsup of dyadic balls whose centers are chosen according to a Bernoulli law.

We continue to consider a compact subset $\I$ of $[0,1]$ such that $$ \dim_{\mathcal{H}} (\I)=\dim_{u-box}(\I)=H>0.$$ Fix $\alpha >0$ and $\eta \in (0,H)$. Recall that the lacunary wavelet series on $\I$ with parameters $(\alpha, \eta)$ is defined by  
$$
f = \sum_{j \in \N} \sum_{k=0}^{2^j-1} c_{j,k} \psi_{j,k } \quad \text{with}\quad
\begin{cases}
2^{-\alpha j} \xi_{j,k }& \text{ if } \lambda_{j,k} \cap \I \neq \emptyset\\
0 & \text{ otherwise,}
\end{cases}
$$
where $(\xi_{j,k})_{j,k}$ denotes a sequence of independent Bernoulli random variables with parameter $2^{(\eta-H)j}$.  
Using Lemma \ref{lem:estimI_j}, we know that the average number of nonzero coefficients at each large scale $ j $ is comparable to $ 2^{\eta j} $. In particular, one might expect that the lowest possible regularity, clearly given by $ \alpha $, is reached on an iso-Hölder set of dimension $ H $.   Moreover, it is clear that points outside $ \I $ have a Hölder exponent equal to the regularity $ r(\psi) $ of the wavelet.

Let us now show that the regularity of ``almost every'' point belonging to $ \I $ is controlled by the lacunarity parameter $ \eta $: the larger $ \eta $, the more regular the function.  For every fixed $b \in (0,1)$, we  will consider as before $J_0$ large enough such that the estimations of the cardinality of $FD(J,b,\eps)$, $ND(J,b,\eps)$ and $C_{b}SD(J,b,\eps)$ given in Proposition \ref{prop:CardSNF} hold for every $J \geq J_0$, and we will consider $k$ large enough so that  $J = (1+b)^k \geq J_0$. The corresponding quasi-Cantor set given in Equation \eqref{eq:K_eps} will then be of the form
$$
\K_{\varepsilon_n}(b) = \bigcap_{p \geq p_0} T_{\infty} \big( (1+b)^p , b_n, \varepsilon_n \big)
$$
for some $p_0>0$. 

\begin{proposition}\label{prop:BC}
Let $n \in \N_0$ and fix $\beta_n= \frac{1}{\eta}\left(H+\frac{1}{n}\right) - 1$. Consider $\ell \in \N$ such that $$(1+\beta_n) = (1+b_n)^\ell$$ for some $b_n>0$ with $b_n^2 < \frac{1}{10n \beta_n}$ and $\beta H - 5 b_n^2>0$. Finally put $\varepsilon_n = b_n^2$ and consider $\K_{\eps_n}(b_n)$ as defined in Equation \eqref{eq:K_eps}.  
Almost surely,  there exists $p_0  \in \N$ such that
$$
\sup_{\lambda' \subset \lambda}|c_{\lambda'}| \geq 2^{- \frac{\alpha}{\eta} \left( H + \frac{1}{n}\right) j }
$$
    for every $\lambda$ of scale $j =\lfloor{(1+\beta_n)^p}\rfloor$ with $p \ge p_0$ and  such that $\lambda \cap \K_{\eps_n}(b_n) \neq \emptyset$. In particular, 
    almost surely
    \begin{enumerate}
        \item $h_f(x) \leq \frac{\alpha }{\eta}(H+ \frac{1}{n} )$ for every $x \in \K_{\eps_n}(b_n)$,
        \item one has $$
    \K_{\eps_n}(b_n) \subset  \bigcap_{p \geq p_0}\bigcup_{k \in F_\eps (p+1)} B(k2^{-(1+\beta_n)^{p+1}}, 2 \cdot 2^{-(1+\beta_n)^p} )
    $$
  where 
    $$
    F_{\eps_n} (p+1) = \{k :  \lambda_{\lfloor(1+\beta_n)^{p+1}\rfloor,k} \cap \K_{\eps_n}(b_n) \neq \emptyset \text{ and } c_{\lfloor(1+\beta_n)^{p+1}\rfloor,k} \neq 0 \}.
    $$
    \end{enumerate}
\end{proposition}

\begin{proof}
Remember that the quasi-Cantor set $\K_{\varepsilon_n}(b_n)$ is obtained as
$$
\K_{\varepsilon_n}(b_n) = \bigcap_{p \geq p_0} T_{\infty} \big( (1+b_n)^p , b_n, \varepsilon_n \big)
$$
where $p_0$ is chosen large enough so that  the Hausdorff dimension of $\K_\varepsilon(b_n)$ is larger than $H-\varepsilon_n$. Using the equality $(1+\beta_n)^p = (1+b_n)^{\ell p}$, we know from Theorem \ref{theo:Vinf} that 
$$
2^{(1+\beta_n)^p (H-\varepsilon_n)} \le \# T_{\infty}((1+\beta_n)^p ,b_n, \eps_n) \le 2^{(1+\beta_n)^p (H+\varepsilon_n)}
$$ 
for all $p \geq p_0/\ell$. Furthermore, we know  from Proposition \ref{prop:bonnereproduction} that  if $\lambda \in T_{\infty} \big( (1+\beta_n)^p , b_n, \varepsilon_n \big)$
$$
\# \big\{ \lambda' \subset \lambda : \lambda' \in T_{\infty}((1+\beta_n )^{p+1},b_n,\eps_n) \big\}\ge 2^{\beta_n(1+\beta_n)^p (H-5 \frac{\eps_n}{b_n})}. 
$$
We consider 
$$
\Omega_p = \big\{ \exists \lambda \in \I_{(1+\beta_n)^p } : \lambda \cap \K_{\varepsilon_n}(b_n) \neq \emptyset \text{ and } \sup_{\lambda' \subset \lambda}|c_{\lambda'}|  < 2^{-\frac{\alpha}{\eta} (H+\frac{1}{n})(1+\beta_n)^p} \big\}.
$$
At the scale $j_0 = (1+\beta_n)^{p+1} $, one has $2^{-\alpha j_0} \geq 2^{-\frac{\alpha }{\eta}(H+\frac{1}{n})(1+\beta_n)^p }$ hence
\begin{align*}
  \mathbb{P}(\Omega_k)  
  & \leq \sum_{\lambda \in \I_{(1+\beta_n)^{p+1} } : \lambda \cap \K_{\varepsilon} (b_n)\neq \emptyset} \mathbb{P}(\forall \lambda' \subset  \lambda, \xi_{\lambda'} = 0) \\
  & \leq 2^{(1+\beta_n)^{p+1}(H+ \eps_n)} (1-2^{(\eta-H)(1+\beta_n)^{p+1} })^{2^{\beta_n(1+\beta_n)^p (H-5 \frac{\eps_n}{b_n})}}\\
  & \leq 2^{(1+\beta_n)^{p+1}(H+ \eps_n)}\exp \left(-2^{(\eta-H)(1+\beta_n)^{p+1} } 2^{\beta_n(1+\beta_n)^p (H-5 \frac{\eps_n}{b_n})}\right)\\
  & \leq 2^{(1+\beta_n)^{p+1}(H+ \eps_n)}\exp \left(-2^{\frac{1}{2n} (1+\beta_n)^{p}}\right)
\end{align*}
by using the choice of $\beta_n$, $b_n$ and $\eps_n$. An application of the Borel-Cantelli Lemma gives that almost surely, for every $p$ large enough and every $\lambda \in \I_{(1+\beta_n)^p }$ such that $\lambda \cap \K_{\varepsilon_n}(b_n) \neq \emptyset $, one has
$$
d_\lambda \geq 2^{-\frac{\alpha}{\eta} (H+\frac{1}{n})(1+\beta_n)^p}.
$$
\end{proof}

\begin{corollary}
Almost surely, for every $\eps>0$, one has
    $$
   \dim_{\mathcal{H}} \big\{ x : h_f(x) \leq \frac{\alpha H}{\eta} + \eps\big\} = H .
    $$
\end{corollary}

    \begin{proof}
    For every $n \in \N$, we know by Proposition \ref{prop:BC} that almost surely
        $$
       \dim_{\mathcal{H}} \big\{ x : h_f(x) \leq \frac{\alpha }{\eta}(H+\frac{1}{n}) \big\} \geq  \dim_{\mathcal{H}}(\K_{\varepsilon_n}(b_n)) \geq H- \eps_n  
        $$
        with $ \eps_n <\frac{1}{10(nH+1)}$. It suffices then to take $n \to + \infty$. 
    \end{proof}

\begin{theorem}\label{theo:LWSlowbound}
Almost surely, for every $h \in [\alpha, \frac{\alpha H}{\eta}]$, one has
    $$
       \dim_{\mathcal{H}} \big\{ x : h_f(x) \leq h \}\ge  h \frac{\eta }{\alpha} .
    $$
\end{theorem}

\begin{proof}
Using Lemma \ref{prop:Edelta}, the result boils down to establishing a lower bound for the Hausdorff dimension of the sets
$$
E_{\delta}(f)= \limsup_{j \to + \infty} \bigcup_{k, \, \xi_{j,k} \neq 0 } B(k2^{-j}, 2^{-\delta (1-\eps_j)j}).
$$
We fix again $n \in \N$, $\beta_n = \frac{1}{\eta}(H+\frac{1}{n})-1$. We consider as previously $\ell \in \N$ such that $$(1+\beta_n) = (1+b_n)^\ell$$ for some $b_n>0$ with $b_n^2 < \frac{1}{10n \beta_n}$ and $\beta H - 5 b_n^2>0$. Finally put $\varepsilon_n = b_n^2$. We consider $\Ken(b_n)$ as defined in Equation \eqref{eq:K_eps}.
Using Proposition \ref{prop:BC}, we can fix $p_0$  large enough such that 
$$
\Ken(b_n) \subset \bigcap_{p\geq p_0}\bigcup_{k \in F_{\epsn}(p+1)} B(k2^{-(1+\beta_n)^{p+1}}, 2\cdot 2^{-(1+\beta_n)^p}).
$$
For $s$ such that  $1\leq s\leq 1+\beta_n$ that can be written as  $s=(1+b_n)^{q_0}$ with $q_0 \in \{0, \dots, \ell\}$, 
we claim and prove below that almost surely, one has
\begin{equation}\label{eq:dimlimsup1}
\dim_{\mathcal{H}} \left(\Ken(b_n) \cap  \bigcap_{p\geq p_0}\bigcup_{k \in F_{\epsn}(p+1)} B(k2^{-(1+\beta_n)^{p+1}},2^{-s(1+\beta_n)^p}) \right) \ge \frac{H-7b_n}{s(1+b_n)}. 
\end{equation}
An immediate consequence is that, for any $\delta \in [\frac{1}{1+\beta_n} , 1]$, if one sets $t=(1+\beta_n){\delta}$ and if $q_0 \in \{0, \dots, \ell\}$ is chosen so that $(1+b_n)^{q_0} \leq t <(1+b_n)^{q_0+1}$, one can write almost surely
\begin{eqnarray*}
\dim_{H}E_{\delta}(f) & \ge & \dim_{\mathcal{H}} \left(\Ken(b_n) \cap  \bigcap_{p\geq p_0}\bigcup_{k \in F_{\epsn}(p+1)} B(k2^{-(1+\beta_n)^{p+1}},2^{-t(1+\beta_n)^p}) \right) \\
& \ge & \dim_{\mathcal{H}} \left(\Ken (b_n) \cap  \bigcap_{p\geq p_0}\bigcup_{k \in F_{\epsn}(p+1)} B(k2^{-(1+\beta_n)^{p+1}},2^{-s(1+\beta_n)^p}) \right)\\
& \ge & 
\frac{H-7b_n}{s(1+b_n)}\\
& \ge & \frac{H-7b_n}{t(1+b_n)^2}\\
& \ge & \frac{H-7b_n}{\delta(1+\beta_n)(1+b_n)^2}\\
& = & \eta \frac{H-7b_n}{\delta(H + \frac{1}{n})(1+b_n)^2}
\end{eqnarray*}
where  $s=(1+b_n)^{q_0}$.
Taking  a dense set $\{h_k : k \in \N \}$ of $[\alpha, \frac{\alpha}{\eta}H]$ and the limit as $n \to +\infty$,  Lemma \ref{prop:Edelta} implies that, almost surely, for any $k \in \N$, 
$$\dim_{\mathcal{H}} \{x \, : \, h_f(x) \le  h_k \} \ge \dim_{\mathcal{H}} E_{\frac{\alpha}{h_k}} (f) \ge  \eta \frac{h_k}{\alpha}.
$$
Since $\dim_{\mathcal{H}} \{x \, : \, h_f(x) \le  h \}$ is an increasing function, it comes that, almost surely, 
$$
\mathcal{D}_{f, \le} (h) = \dim_{\mathcal{H}} \{x \, : \, h_f(x) \le  h \}  \ge  \eta \frac{h}{\alpha}, \quad \forall h \in [\alpha, \frac{\alpha}{\eta}H]
$$
and it will yield the stated result.

To obtain the lower bound of the dimension stated in Equation \eqref{eq:dimlimsup1}, as is often done (see, for instance, \cite{Jaffard:00,BucSeu:10}), we will construct a Cantor set inside
\[
\Ken (b_n)\cap  \bigcap_{p\geq p_0}\bigcup_{k \in F_{\epsn}(p+1)}  B(k2^{-(1+\beta_n)^{p+1}},2\cdot 2^{-s(1+\beta_n)^p}) 
\]
and a measure on the Cantor set in order to apply the Mass distribution principle recalled below.

\begin{lemma}[Mass distribution principle \cite{Falconer:86}]\label{lem:MDP}
Let $F$ be a Borel subset of $\R$ and $\mu$ be a mass distribution on $F$. Suppose that for some $s$ there exists $c>0$ and $\delta>0$ such that
$$
\mu(U) \le c \vert U \vert^{t}
$$
for all sets $U$ with $\vert U \vert \le \delta$. Then $\mathcal{H}^s(F) \ge \mu(F)/c$ and in particular, 
$$
t \le \dim_{\mathcal{H}} F.
$$
\end{lemma}

The Cantor set and the measure are constructed by an iterative process that we describe below. We fix $s$ such that $1\le s \le 1+\beta_n$ and $s=(1+b_n)^{q_0}$ with $ q_0\in \{0, \dots \ell\}$.\\

\noindent {\bf First generation :} Fix $p_1 \geq p_0$. Since 
$$
\Ken(b_n) \subset \bigcup_{k \in F_{\varepsilon_n}(p_1+1)} B(k2^{-(1+\beta_n)^{p_1+1}}, 2\cdot 2^{-(1+\beta_n)^{p_1}})
$$
there are at least $2^{(1+\beta_n)^{p_1}(H - \eps_n)}$ elements in $F_{\varepsilon_n}(p_1+1)$ since
$$
\# T_{\infty} ((1+\beta_n)^{p_1}, b_n, \eps_n) \geq 2^{(1+\beta_n)^{p_1} (H- \eps_n)}
$$
by Theorem \ref{theo:Vinf}. We select exactly $2^{(1+\beta_n)^{p_1}(H- 6 \frac{\varepsilon_n}{b_n})} = 2^{(1+\beta_n)^{p_1}(H- 6b_n)}$ elements corresponding to   dyadic points $x^1_r =k_r2^{-(1+\beta_n)^{p_1+1}}$, with $1 \le r \le 2^{(1+\beta_n)^{p_1}(H-6b_n)}$ and so that the balls 
$$
 B(x^1_r, 2^{-s(1+\beta_n)^{p_1}}), \quad 1 \le r \le 2^{(1+\beta_n)^{p_1}(H-6b_n)},
$$

are disjoints two by two. Note that this choice is possible if $p_1$ is large enough. We set
$$
B^1_r =  B(x^1_r, 2^{-(1+b_n)^{q_0}(1+\beta_n)^{p_1}}), \quad 1 \le r \le 2^{(1+\beta_n)^{p_1}(H-6b_n)}.
$$
These balls give the first generation of our Cantor set and the  mass of $\mu$ is uniformly distributed on the balls, so
$$
\mu (B^1_r)= 2^{- (H-6b_n)(1+\beta_n)^{p_1}}= \vert B^1_r \vert^{(H-6b_n)/(1+b_n)^{q_0}}=\vert B^1_r \vert^{(H-6b_n)/s}
$$
for all $1 \le r \le 2^{(1+\beta_n)^{p_1}(H-6b_n)}$.
Note that by construction, 
\begin{equation}\label{eq:step1}
B^1_r \in T_{\infty} ((1+b_n)^{q_0}(1+\beta_n)^{p_1}, b_n, \varepsilon_n).
\end{equation}\\

\noindent
{\bf Choice of the scales :} The next generations of the Cantor set will be constructed at   scales $(1+b_n)^{q_0}(1+\beta_n)^{p_N}$, with $N\ge 1$, where the sequence $(p_N)_N$ of integers is chosen such that for all $N \ge 2$
\begin{equation}\label{seqP_n}
2^{(1+\beta_n)^{p_{N-1}} ((1+b_n)^{q_0}-1)(H - 6{b_n})} < 2^{(1+\beta_n)^{p_N} {b_n}2^{-N}}.
\end{equation}\\

\noindent {\bf Second generation :} Let us pick one of the ball  $B^1_{r_1}$ of the first generation. Using \eqref{eq:step1} and the self-similarity given in Proposition \ref{prop:bonnereproduction} (with $\eps_n=b_n^2$),  we know that, for all the scales of the form $(1+b_n)^p$  between $(1+b_n)^{q_0+1}(1+\beta_n)^{p_1}$ and $(1+\beta_n)^{p_2}$, the ball $B^1_{r_1}$ contains at least 
$$2^{((1+b_n)^{p} - (1+b_n)^{q_0}(1+\beta_n)^{p_1}) (H - 5 {b_n})}$$
dyadic intervals of $K_{\eps_n}(b_n)$. We select exactly $2^{((1+b_n)^{p} - (1+b_n)^{q_0}(1+\beta_n)^{p_1}) (H - 6 {b_n})}$ such intervals. At the last scale $p=p_2$ each of these intervals corresponds to a ball $B(x^2_r, 2^{-(1+\beta_n)^{p_2}})$ where $x^2_r = k2^{(1+\beta_n)^{p_2+1}}$ belongs to $F_{\eps_n}(p_2+1)$, since 
$$
\Ken(b_n) \subset \bigcup_{k \in F_{\varepsilon_n}(p_2+1)} B(k2^{-(1+\beta_n)^{p_2+1}}, 2\cdot 2^{-(1+\beta_n)^{p_2}}).
$$
Our second generation of balls in $B^1_{r_1}$ is then given by $2^{((1+\beta _n)^{p_2} - (1+b_n)^{q_0}(1+\beta_n)^{p_1}) (H - 6 {b_n})}$ disjoint balls  $B(k2^{-(1+\beta_n)^{p_2+1}}, 2^{-(1+b_n)^{q_0}(1+\beta_n)^{p_2}})$ included in $B^1_{r_1}$. 

\begin{remark} Note that for each dyadic interval of $\K_{\eps}(b_n)$ of scale $(1+\beta_n)^{p_2}$,  we select only one ball $B(x^2_r, 2^{-(1+\beta_n)^{p_2}})$ with $x^2_r = k2^{(1+\beta_n)^{p_2+1}} \in F_{\eps_n}(p_2+1)$. Therefore, the balls with the same center but different dilated radii can intersect at most pairwise. The selection of exactly $2^{((1+\beta _n)^{p_2} - (1+b_n)^{q_0}(1+\beta_n)^{p_1}) (H - 6 {b_n})}$ ensures the possibility of choosing a collection of disjoint balls.
\end{remark}

We label these balls $$B^2_{r_1,r_2}  \quad \text{ with } \quad 1 \le r_2 \le 2^{((1+\beta _n)^{p_2} - (1+b_n)^{q_0}(1+\beta_n)^{p_1}) (H - 6 {b_n})}.$$
Next, we uniformly distribute the mass of $B^1_{r_1}$ on the disjoint smaller balls, so that 
\begin{align*}
\mu(B^2_{r_1,r_2})&= \frac{\mu(B^1_{r_1})}{2^{((1+\beta _n)^{p_2} - (1+b_n)^{q_0}(1+\beta_n)^{p_1}) (H - 6 b_n)}} \\
& = \frac{2^{(1+\beta_n)^{p_1}(H-6b_n)}}{2^{((1+\beta _n)^{p_2} - (1+b_n)^{q_0}(1+\beta_n)^{p_1}) (H - 6 b_n)}}\\
& \le 2^{-(1+\beta_n)^{p_2}(H-6{b_n}- \frac{b_n}{4})} \\
& = \vert B^2_{r_1,r_2}\vert^{\frac{H-6b_n- \frac{b_n}{4}}{(1+b_n)^{q_0}}}\\
& = \vert B^2_{r_1,r_2}\vert^{\frac{H-6b_n- \frac{b_n}{4}}{s}},
\end{align*}
since $p_2$ is chosen large enough so that 
$$
2^{(1+\beta_n)^{p_1} ((1+b_n)^{q_0}-1)(H - 6b_n)} < 2^{(1+\beta_n)^{p_2} \frac{b_n}{4}}.
$$\\

\noindent {\bf $N^\text{th}$ generation : }  We assume that the balls $B^{N-1}_{r_1, \dots, r_{N-1}}$ and the measure $\mu$ have been constructed so that
$$
\mu (B^{N-1}_{r_1, \dots, r_{N-1}}) \leq 2^{-(1+\beta_n)^{p_{N-1}}(H-{b_n}(6+\sum_{m = 2}^{N-1} 2^{-m} ))}  = \vert B^{N-1}_{r_1, \dots, r_{N-1}}\vert^{\frac{H-{b_n} (6+\sum_{m = 2}^{N-1} 2^{-m} )}{(1+b_n)^{q_0}}}.
$$ Let us consider one of this balls. It contains at least
$$2^{((1+\beta _n)^{p_N} - (1+b_n)^{q_0}(1+\beta_n)^{p_{N-1}})(H - 5 b_n)}$$
dyadic intervals of $K_{\eps_n}(b_n)$ at scale $(1+\beta_n)^{p_N}$. Again, each of these intervals corresponds  to a ball $B(x^N_r, 2^{-(1+\beta_n)^{p_N}})$ where $x^N_r = k2^{(1+\beta_n)^{p_N+1}}$ belongs to $F_{\eps_n}(p_N+1)$.
Our $N^\text{th}$ generation of balls in $B^N_{r_1, \dots, r_N}$ is then given by $2^{((1+\beta _n)^{p_N} - (1+b_n)^{q_0}(1+\beta_n)^{p_{N-1}}) (H - 6 {b_n})}$  disjoint balls  $B(k2^{-(1+\beta_n)^{p_N+1}}, 2^{-(1+b_n)^{q_0}(1+\beta_n)^{p_N}})$ included in $B^{N-1}_{r_1, \dots, r_{N-1}}$. Again, this choice can be done such that at all intermediary scales $(1+b_n)^p$, exactly $2^{((1+b_n)^p - (1+b_n)^{q_0}(1+\beta_n)^{p_{N-1}})(H - 6 b_n)}$ dyadic intervals of $\K_{\eps_n}(b_n)$ are selected. 
We label this balls $$B^N_{r_1,\dots, r_N}  \quad \text{ with } \quad 1 \le r_N \le 2^{((1+\beta _n)^{p_N} - (1+b_n)^{q_0}(1+\beta_n)^{p_{N-1}}) (H - 6 b_n)}.$$
Next, we uniformly distribute the mass of $B^{N-1}_{r_1, \dots, r_{N-1}}$ on the disjoint smaller balls. It follows that 
\begin{eqnarray*}
\mu (B^N_{r_1,...,r_N})& =&  \frac{\mu(B^{N-1}_{r_1, \dots, r_{N-1}})}{2^{((1+\beta _n)^{p_N} - (1+b_n)^{q_0}(1+\beta_n)^{p_{N-1}}) (H - 6 {b_n})}}\\
&= & \frac{2^{-(1+\beta_n)^{p_{N-1}}(H-b_n(6+\sum_{m = 2}^{N-1} 2^{-m} ))}}{2^{((1+\beta _n)^{p_N} - (1+b_n)^{q_0}(1+\beta_n)^{p_{N-1}}) (H - 6 {b_n})}}\\
&\leq & 2^{-(1+\beta_n)^{p_{N}}(H-b_n(6+ \sum_{m = 2}^{N} 2^{-m} ))}  \\
& =  & \vert B^{N}_{r_1, \dots, r_{N}}\vert^{\frac{H-b_n(6+ \sum_{m = 2}^{N} 2^{-m}) }{(1+b_n)^{q_0}}} \\
& =  & \vert B^{N}_{r_1, \dots, r_{N}}\vert^{\frac{H-b_n(6+ \sum_{m = 2}^{N} 2^{-m}) }{s}}.
\end{eqnarray*}
because of \eqref{seqP_n}.\\

\noindent {\bf Defintion of the Cantor set:}  We define the Cantor set $\Zen (b_n) \subset \Ken (b_n)$ by
$$
\Zen (b_n) = \bigcap_N \bigcup_{(r_1,...,r_N)} B^{N}_{r_1,...,r_N} 
$$
The support of $\mu$ obtained by the distribution process above is included in $\Zen (b_n)$ (see \cite{Falconer:86}, Prop. 1.7) and for any ball $B$ appearing in the construction of the Cantor set, we have
\begin{equation}\label{eq:mesballcantor}
\mu(B) \le \vert B \vert^{\frac{H-7b_n}{s}}. 
\end{equation}

\medskip

\noindent {\bf Measure of a ball of  $\Ken (b_n)$:} Assume first that  $D$ correspond to a dyadic interval that appears in the construction of  $\Ken (b_n)$ that intersects $\Zen (b_n)$. Let  $N$ be such
that
$$
2^{-(1+\beta_n)^{p_{N+1} } } < |D| \leq 2^{-(1+\beta_n)^{p_N } }.
$$
We consider $m\in \{0, \dots, q\}$ and $p \in \{ p_N, \dots, p_{N+1}-1\}$ such that
$$
|D|  = 2^{-(1+\beta_n)^{p }(1+b_n)^m  }
$$
We consider two cases:
\begin{itemize}
    \item First, assume that 
    $$
   2^{-(1+\beta_n)^{p_N }(1+b_n)^{q_0}  } \leq |D| \leq 2^{-(1+\beta_n)^{p_N }  },
    $$
i.e. $p=p_N$ and $m \leq q_0$. Then, there exists at most one ball $B^{N}_{r_1, \dots, r_N}$ of the $N^{\text{th}}$ generation of $\Zen(b_n)$ such that $B^{N}_{r_1, \dots, r_N}\cap D \neq \emptyset$. Hence
$$
\mu(D) \leq \mu (B^{N}_{r_1, \dots, r_N}) \leq |B|^{\frac{H-7{b_n}}{(1+b_n)^{q_0}}} \leq |B|^{\frac{H-7{b_n}}{s} }\leq |D|^{\frac{H-7{b_n}}{s}}. 
$$

\item Next, assume that
$$
 |D|=   2^{-(1+\beta_n)^{p_N }(1+b_n)^{q_0+m} } 
$$
with $m \geq 1$ and $qp_N + q_0+ m < p_{N+1}$. 
Since $D \in T_{\infty} ((1+\beta_n)^{p_N }(1+b_n)^{m} , b_n, \eps_n)$, 
we know  that there are at most 
$$
2^{((1+\beta_n)^{p_{N+1} } - (1+\beta)^{p_N}(1+b_n)^{q_0+m}) (H - 6 {b_n})}
$$
children of $D$ at scale $(1+\beta_n)^{p_{N+1} }$ in $D$. Since each of these intervals contains at most one subinterval of the $(N+1)^{\text{th}}$ generation of the Cantor set, it follows that
\begin{eqnarray*}
\mu(D) & \leq & 2^{((1+\beta_n)^{p_{N+1} } - (1+\beta)^{p_N}(1+b_n)^{q_0+m}) (H -6{b_n})}   \frac{\mu(B^{N})}{2^{((1+\beta _n)^{p_{N+1}} - (1+b_n)^{q_0}(1+\beta_n)^{p_{N}}) (H - 6 {b_n})}}\\
&\leq &  2^{- (1+\beta_n)^{p_N}(1+b_n)^{q_0+m}(H -6 {b_n})}   \frac{\mu(B^{N})}{2^{ - (1+b_n)^{q_0}(1+\beta_n)^{p_{N}} (H - 6 {b_n})}}\\
&\leq &  2^{-(1+b_n)^{q_0}((1+b_n)^{m}-1)(1+\beta_n)^{p_{N}}(H -6 {b_n})}   \mu(B^{N})\\
&\leq & 2^{-((1+b_n)^{m}-1)(1+\beta_n)^{p_{N}}(H -6 {b_n})}  2^{-(1+\beta_n)^{p_{N}}(H-7{b_n} )}\\
&\leq & 2^{- (1+b_n)^{m} (1+\beta_n)^{p_N} (H - 7 {b_n})}\\
& \leq & |D|^{\frac{(H - 7 {b_n})}{(1+b_n)^{q_0}}}\\
& = & |D|^{\frac{(H - 7 {b_n})}{s}}.
\end{eqnarray*}

\end{itemize}

\medskip

\noindent {\bf {Measure of any ball $D$} of radius smaller than $2^{-(1+\beta)^{p_1}}$:} Clearly, if 
$D \cap \Zen = \emptyset$, then $\mu(D)=0$. Hence, we  can assume that $D \cap \Zen \neq \emptyset$. If $\vert D \vert = 2^{-\tau}$, we consider $N \ge 1$, $m\in \{0, \dots, q\}$ and $p \in \{ p_N, \dots, p_{N+1}-1\}$ such that
$$
(1+\beta_n)^{p }(1+b_n)^{m+1}   <\tau < (1+\beta_n)^{p }(1+b_n)^{m}.
$$ 
Such a ball intersects at most two balls $B_1,B_2$ of $\Ken(b_n)$ of size $2^{-(1+\beta_n)^{p }(1+b_n)^{m}}$. It follows that
\begin{eqnarray*}
\mu(D) &\le & \mu(B_1)+\mu(B_2) \\
& \le & 2\cdot 2^{-(1+\beta_n)^{p }(1+b_n)^{m} \frac{H-7b_n}{s}}\\
& \le & 2\cdot 2^{-\tau \frac{H-7b_n}{s(1+b_n)}}\\
& = & 2 \cdot \vert D \vert^{\frac{H-7b_n}{s(1+b_n)}}.
\end{eqnarray*}

\medskip

\noindent {\bf{Dimension of $\Zen$:}} We conclude that $\dim_{\mathcal{H}}(\Zen)\ge \frac{H-7b_n}{s(1+b_n)}$ by applying the Mass Ditribution Principle recalled in Lemma \ref{lem:MDP}. his establishes \eqref{eq:dimlimsup1}, and thus completes the proof.
 \end{proof}

\begin{proposition}\label{prop:dimisoholder}
    Almost surely, for every $h \in [\alpha, \frac{\alpha H}{\eta}]$, one has
    $$
       \dim_{\mathcal{H}} \big\{ x : h_f(x) \leq h \} \le h \frac{\eta }{\alpha} .
    $$
\end{proposition}

\begin{proof}
Let us fix $h \in [\alpha, \frac{\alpha H}{\eta}]$.
Fix $\eps>0$  and remark that for every $j$ large enough, one has 
$$
\mathbb{E} [\# \{k : c_{j,k} =2^{-\alpha j} \}] \leq  2^{(\eta-H)j }  2^{j(H+\varepsilon)} = 2^{(\eta + \varepsilon)j}
$$
by applying Lemma \ref{lem:estimI_j}. Chebychev inequality combined with Borel-Cantelli lemma gives directly that almost surely, for $j$ large enough, 
$$
\# \{k : c_{j,k} =2^{-\alpha j} \} \leq 2^{(\eta + 2 \varepsilon)j}.
$$
Considering a sequence $(\eps_n)_n$ that decreases to $0$, we obtain that
$$
\rho_f(\alpha) := \limsup_{j \to + \infty} \frac{\log \# \{k : c_{j,k} =2^{-\alpha j} \} }{\log 2^j} \leq \eta. 
$$
almost surely. It has been proved in \cite{Jaffard:02} that
$$
{\mathcal{D}}_f(h) \leq h \sup_{h' \in ( 0,h]}  \frac{\rho_f(h')}{h'}
$$
The right term is an increasing function of $h$, hence it is also an upper-bound of ${\mathcal{D}}_{f, \le}(h)$. 

Since $\rho_f(h') = - \infty$ if $h' \neq \alpha$, we obtain that
$$
 \dim_{\mathcal{H}} \big\{ x : h_f(x) \leq h \} \leq h \frac{\eta}{\alpha }.
$$
\end{proof}

The combination of Theorem \ref{theo:LWSlowbound} and Proposition \ref{prop:dimisoholder} finally yields Theorem \ref{thm:spectreLWS}.

\section{Conclusion and comments}

Let us conclude by presenting an example that demonstrates the optimality of our results given the tools developed in the paper. We consider a sequence $(K_n)_{n \ge 1}$ of disjoint Cantor sets of Hausdorff  dimension $1-\frac{1}{n}$, and let us set
$$
\I =  \bigcup_{n \ge 1} K_n.
$$
Clearly, $\I$ satisfies  the upper-box counting formalism with
$$
\dim_{\mathcal{H}}(\I) = \dim_{u-box}(\I) = 1. 
$$ 
Now, consider the lacunary wavelet series $f$ on $\I$. Since the Cantor sets are disjoint, one has
$$
\{x : h_f(x) = h \} = \bigcup_{n \in \N} \{x \in K_n: h_{f_n}(x) =h \}
$$
where $$
f_n = \sum_{j \in \N} \sum_{k=0}^{2^j-1} c_{j,k} \psi_{j,k } \quad \text{with}\quad
\begin{cases}
2^{-\alpha j} \xi_{j,k }& \text{ if } \lambda_{j,k} \cap K_n \neq \emptyset\\
0 & \text{ otherwise,}
\end{cases}
$$
is the corresponding lacunary wavelet series on $K_n$. Note that $\xi_{j,k}$  is a Bernoulli random variable with parameter $2^{(\eta-1)j}= 2 ^{(\eta-\frac{1}{n} - \dim_\mathcal{H}(K_n)) j}$.  Hence, the parameter of lacunarity of $f_n$ is given by $\eta-\frac{1}{n}$.
From \cite{EV23}, we know that if $\eta-\frac{1}{n}>0$, one has almost surely
$$
\dim_{\mathcal{H}}\{x \in K_n: h_{f_n}(x) =h \} = h \frac{\eta-\frac{1}{n}}{\alpha}
$$
for all $h \in [\alpha, \frac{\alpha (1- \frac{1}{n})}{\eta-\frac{1}{n}}]$, and the iso-H\"older set is empty for all other finite values of $h$. Consequently,
$$
\dim_{\mathcal{H}} \{x : h_f(x) = h \} = \sup_{n \ge 1}  h \frac{\eta-\frac{1}{n}}{\alpha} = h \frac{\eta}{\alpha}
$$
for all $h \in [\alpha, \frac{\alpha H}{\eta}]$. Let us however make two remarks.
\begin{itemize}
    \item First, note that 
    $$
    \mathcal{H}^{h \frac{\eta}{\alpha}}(\{x : h_f(x) = h \}) = \sum_{n \in \N}  \mathcal{H}^{h \frac{\eta}{\alpha}}(\{x : h_{f_n}(x) = h \}) = 0.
    $$
    Hence, in general, one cannot expect that the $h\frac{\eta}{\alpha}$-dimensional Hausdorff measure of the iso-H\"older set of level $h$ of a LWS $f$ is strictly positive. 

    \item One cannot hope to obtain precise results about the regularity that exceed the regularity attained ``almost everywhere'' on $\I$. Indeed, remark that for all $n$, one has
    $$
    \frac{\alpha (1- \frac{1}{n})}{\eta-\frac{1}{n}} > \frac{\alpha }{\eta}
    $$
    so that the maximal regularity on each $K_n$ is larger than the regularity $\frac{\alpha }{\eta}$ attained by $f$ on a set of dimension $1$. 
\end{itemize}

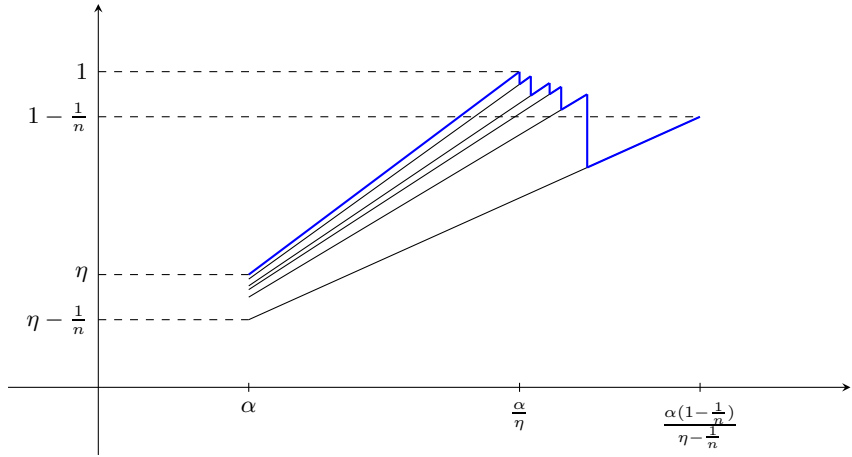
\begin{figure}[ht]
\centering
\begin{tikzpicture}
\begin{axis}[
x=4cm,y=6cm,
axis lines=middle,
xmin=-0.3,
xmax=2.5,
ymin=-0.15,
ymax=0.85,
xtick=\empty,
ytick = \empty,
]
\draw[thick, color=blue, samples=100,domain=0.5:1.4] 
 {plot ({\x},{\x/2})};
\draw  (0.5,-0.01) -- (0.5,0.01);
\draw  (0.5,-0.01) [below]node {\footnotesize$\alpha$};
\draw  (1.4,-0.01) -- (1.4,0.01);
\draw  (1.4,-0.01) [below]node {\footnotesize$\frac{\alpha }{\eta}$};
\draw[dashed] (0,0.7) -- (1.4,0.7);
\draw  (0,0.7) [left]node {\footnotesize$1$};
\draw[dashed] (0,0.25) -- (0.5,0.25);
\draw  (0,0.25) [left]node {\footnotesize$\eta$};
\draw[color=black, samples=100,domain=0.5:2] 
 {plot ({\x},{\x*0.3})};
 \draw[dashed] (0,0.6) -- (2,0.6);
\draw  (0,0.6) [left]node {\footnotesize$ 1 - \frac{1}{n}$};
\draw  (2,-0.01) -- (2,0.01);
\draw  (2,-0.01) [below]node {\footnotesize$\frac{\alpha (1-\frac{1}{n})}{\eta - \frac{1}{n}}$};
\draw[dashed] (0,0.15) -- (0.5,0.15);
\draw  (0,0.15) [left]node {\footnotesize$\eta- \frac{1}{n}$};
\draw[color=black, samples=100,domain=0.5:13/8] 
 {plot ({\x},{\x*2/5})};
 \draw[color=black, samples=100,domain=0.5:20/13] 
 {plot ({\x},{\x*13/30})};
  \draw[color=black, samples=100,domain=0.5:3/2] 
 {plot ({\x},{\x*9/20})};
   \draw[color=black, samples=100,domain=0.5:69/48] 
 {plot ({\x},{\x*24/50})};
 \draw[thick, color=blue]   (1.4,1.4*24/50) -- (1.4,0.7);
    \draw[thick, color=blue, samples=100,domain=1.4:69/48] 
 {plot ({\x},{\x*24/50})};
  \draw[thick, color=blue]   (69/48,69/48*9/20) -- (69/48,69/48*24/50);
    \draw[thick, color=blue, samples=100,domain=69/48:3/2] 
 {plot ({\x},{\x*9/20})};
   \draw[thick, color=blue]   (3/2,3/2*9/20) -- (3/2,3/2*13/30);
    \draw[thick, color=blue, samples=100,domain=3/2:20/13] 
 {plot ({\x},{\x*13/30})};
   \draw[thick, color=blue]   (20/13,20/13*13/30) -- (20/13,20/13*2/5);
    \draw[thick, color=blue, samples=100,domain=20/13:13/8] 
 {plot ({\x},{\x*2/5})};
    \draw[thick, color=blue]   (13/8,13/8*2/5) -- (13/8,13/8*0.3);
    \draw[thick, color=blue, samples=100,domain=13/8:2] 
 {plot ({\x},{\x*0.3})};
\end{axis}

\end{tikzpicture}
\caption{The regularity of a LWS on $\I$ can be larger than the regularity attained on a subset of the dimension of $\I$. }
\end{figure}
\bigskip

The strategy of constructing quasi-Cantor sets with self-similarity properties at prescribed scales inside the fractal $\I$ allow us to recover some Mass Transference Principles for limsup of balls whose centers are chosen according to Bernoulli laws. The counterpart is that we do not obtain the preservation of the positivity of a  measure as mentioned  above, but only an estimation of the dimension. The usual arguments to obtain the multifractal spectrum, and not just the increasing one, are crucially based on the existence of a positive measure on the sets $E_{\delta}(f)$. Indeed, for a  LWS $f$ of parameter $(\alpha, \eta)$, and every  $\delta \in (0,1)$, we have
\[
\{ x \in \I : \, h_f(x)= \frac{\alpha}{\delta} \}= \I \cap\left( \bigcap_{0<\delta' < \delta} E_{\delta'}(f)
\setminus \bigcup_{\delta < \delta' \leq 1} E_{ \delta'}(f)\right). 
\]
The example presented above tends also to show that very various behaviors can be expected for the regularities greater than the first regularity $h_0$ for which the dimension of the fractal set is reached. We do not know whether the multifractal spectrum of a LWS supported on a fractal set $\I$ might be discontinuous even for regularities smaller than $h_0$. It would be interesting to exhibit LWS on fractal sets $\I$ for which $\dim_{\mathcal{H}}\I$ is not reached - maybe by considering LWS on the attractor of an IFS which do not satisfy the OSC. The difficulty to obtain the multifractal spectrum in situations where the increasing spectrum is known also occurs in the context of  multivariate multifractal analysis. This is the case for the bivariate analysis of two independent LWS on $[0,1]$ (see \cite{JaffSeu:19}). Note also that a relaxed version of the singularity spectrum which is more accurate than the increasing one has been defined the following way:
$$
\widetilde{D}_f(h)= \lim_{\eps \to 0^+} \dim_{\mathcal{H}}\{ x  : \, h-\eps \le h_f(x) \le h+\eps\}
$$
which is easier to obtain since it is enough to compute the exact dimensions of the sets $E_{\delta}(f)$ and for which the validity of the large deviation methods are more robust.

We also expect that our strategy of first constructing quasi-Cantor sets before estimating the dimension of the limsup of balls could be developed in several directions. In particular, in future works, we plan to
\begin{itemize} 
    \item[$\bullet$] Consider other types of random processes on fractal sets, such as random wavelet series or lacunarized cascades;
\item[$\bullet$] Investigate the optimal conditions on the fractal set $\I$ under which we can construct quasi-Cantor sets inside. In particular, we plan to explore the generalization of this method to iso-H\"older sets, for which the upper-box counting dimension is $1$ as soon as there is some homogeneity. This is the case for random cascades or LWS. Such developments would be particularly useful for developing criteria to test the validity of the formalism (see \cite{EV23}) and for multivariate analysis of functions.
\end{itemize}

\noindent{\bf Acknowledgments.} We warmly thank Edouard Daviaud for his enriching discussions, which greatly contributed  to bringing additional insights  to this paper.

\bibliographystyle{plain}
\bibliography{CantorSet.bib}

\newpage 
\appendix

\section{Proofs of the results of Section \ref{sec:goodrate} }
In this appendix,  for pedagogical purposes, we recall the proof of the two results of \cite{EV23} recalled in Section \ref{sec:goodrate}  on which the construction presented in Section \ref{sec:Keps} is based.

\medskip
\noindent\textbf{Proof of Lemma \ref{lem:estimI_j}}
Because the Hausdorff dimension and the upper-box dimension of $\I$ are equal, we have $\log_2 \# \I_j \le H+ \varepsilon$, {\it i.e.} $\#I_j \le 2^{j(H+\varepsilon)}$ for $j$ large enough. 

The lower bound is proved by contradiction. Assume that there is a sequence $(j_n)_n$ such that $\# \I_{j_n} \le 2^{j_n(H-\varepsilon)}$. Since $\I \subset  \widetilde{\I_{j_n}}$, it provides a covering of $\I$ by less than $2^{j_n(H-\varepsilon)}$ intervals of size $2^{-j_n}$. It follows that  $\text{dim}_H(\I) \le H- \varepsilon$, which gives the contradiction.
\null\hfill $\Box$\par\medskip

\medskip

\noindent\textbf{Proof of Proposition \ref{prop:CardSNF}}
Since $\I_j = ND(j,\beta,\eps) \cup SD(j,\beta,\eps) \cup FD(j,\beta,\eps)$, there exists $J  \in \N$, such that for all $j_n \ge J$
$$
 \# \left( SD(j_n,\beta,\eps) \cup FD(j_n,\beta,\eps) \right) \le 2^{j_n(H+\varepsilon)}.
$$
Since $\# \I_{\lfloor (1+ \beta) j \rfloor}$ is bounded by $2^{(1+\beta)(H+\varepsilon)j}$, we obtain that
$$
\# FD(j,\beta,\eps) \times 2^{j(\beta H+4 m \epsilon)} \le \# \I_{\lfloor (1+ \beta) j \rfloor} \le 2^{j(1+\beta)(H+\varepsilon)}.
$$
With $m=\max(1,\beta)$, it gives \eqref{Eq:CardFD}.

Since $\# SD(j,\beta,\eps)$ is bounded by $2^{j(H+\varepsilon)}$ and since the dyadic intervals of $SD(j,\beta,\eps)$ present a slow duplication, we can control the cardinal of $C_{\beta}SD(j,\beta,\eps)$ by
\begin{equation*}\label{eq:CSD}
    \# C_{\beta}SD(j,\beta,\eps) \le 2^{j(H+\varepsilon)} 2^{j (H\beta - 4 m \varepsilon) } \le 2^{j((1+\beta) H-3 m \varepsilon )}
\end{equation*}
which is \eqref{Eq:CardSD}.

We now turn to the set $ND(j,\beta,\eps)$. Its upper bound is trivial. Suppose now that there is a sequence $(j_n)_n$ such that, for all $n$,  
\begin{equation}\label{eq:ND}
    \# ND(j_n,\beta,\eps) \le 2^{j_n(H- 2\varepsilon)}. 
\end{equation}

For each $n \ge 0$, we have the following covering of $\I$ with sets of diameter smaller than $2^{-j_n}$ :
$$
\I \subset \{ \lambda \in ND(j_n,\beta,\eps)\} \cup \{ \lambda \in  FD(j_n,\beta,\eps)\}  \cup \{ \lambda' \in  C_{\beta}SD(j_n,\beta,\eps)\}.
$$
Together with \eqref{Eq:CardND}, \eqref{Eq:CardFD}  and \eqref{Eq:CardSD}, it implies that, for any $r>0$ and $0 <s<1$, 
$$
{\mathcal{H}}^s_r(\I) \le 2 \times 2^{j_n(H-2\varepsilon)} 2^{-j_n s}+ 2^{j_n((1+\beta) H-3 m \varepsilon)} 2^{-(1+\beta)j_n s}
$$
for any $j_n$ such that $2^{-j_n} \le r$. Taking $s = H - \varepsilon$, it comes
$$
{\mathcal{H}}^s_r(\I) \le 2 \times 2^{-\varepsilon j_n} + 2^{-j_n m \varepsilon}
$$
It follows that $\lim_{r \to 0^+} {\mathcal{H}}^s_r(\I) =0$ and therefore $\dim_{\cal I}(\I) \le s- \varepsilon$, which is impossible. Hence, there exists some $J \in \N$ such that for all $j \ge J$, $2^{j(H-2\varepsilon)} \le \# ND(j,\beta,\eps) \le 2^{j(H+\varepsilon)}$.
\null\hfill $\Box$\par\medskip

\end{document}